\documentclass[reqno]{amsart}
\usepackage[backend=bibtex, sorting=none, bibstyle=alphabetic, citestyle=alphabetic, sorting=nyt, maxbibnames=99, giveninits=true, isbn=false, url=false, doi=false, eprint=false]{biblatex}  %reference manager
\renewbibmacro{in:}{}
\bibliography{references.bib}
\usepackage{amsmath, amssymb, calrsfs, graphics, graphicx, enumerate, enumitem, url, xcolor, hyperref, comment, tablefootnote, listings, multirow}
\usepackage[ruled, lined, linesnumbered, longend]{algorithm2e}
\usepackage[justification=centering]{caption}
\usepackage{tikz}
\usetikzlibrary{positioning, arrows, decorations.markings, calc}
\newtheorem{theorem}{Theorem}[section]
\newtheorem{lemma}[theorem]{Lemma}

\newtheorem{corollary}[theorem]{Corollary}
\newtheorem{definition}[theorem]{Definition}
\numberwithin{equation}{section}

\DeclareMathAlphabet{\mathcal}{OMS}{cmsy}{m}{n}

\theoremstyle{remark}
\newtheorem*{remark}{Remark}

\title{Toward optimal exponent pairs}
\author{Timothy S. Trudgian and Andrew Yang}
\date\today
\keywords{Exponent pairs, exponential sums, Riemann zeta-function, Generalised Dirichlet divisor problem}
\subjclass[2020]{Primary: 11L07, 11M06, 11T23}

\address{
    TT: School of Science, University of New South Wales (Canberra), Northcott Drive, Campbell, ACT 2600, Australia
}
\email{timothy.trudgian@unsw.edu.au}
\address{
    AY: School of Science, University of New South Wales (Canberra), Northcott Drive, Campbell, ACT 2600, Australia
}
\email{andrew.yang1@unsw.edu.au}

\begin{document}

\begin{abstract}
We quantify the set of known exponent pairs $(k, \ell)$ and develop a framework to compute the optimal exponent pair for an arbitrary objective function. Applying this methodology, we make progress on several open problems, including bounds of the Riemann zeta-function $\zeta(s)$ in the critical strip, estimates of the moments of $\zeta(1/2 + it)$ and the generalised Dirichlet divisor problem. 
\end{abstract}

\maketitle

\section{Introduction}
Many well-known problems in analytic number theory reduce to bounding exponential sums of the form
\[
\sum_{a < n \le b}e(f(n)),\qquad e(x) := e^{2\pi i x},
\]
where $n$ takes integer values, and $f(n)$ is a function possessing certain smoothness properties. For instance, exponential sum estimates have been used to bound the Riemann zeta-function $\zeta(s)$ in the critical strip, and to estimate the error term in the Dirichlet divisor problem. The deep theory of exponent pairs arose from the study of such exponential sums. Recent advances in the Bombieri--Iwaniec method and Vinogradov's mean value theorem have led to a combinatorial explosion in the number of known exponent pairs. As such, choosing the optimal exponent pair for a given application has become a highly non-trivial problem.

In this article we study the geometry of known exponent pairs and develop a method to compute numerically the optimal exponent pair for arbitrary objective functions. Our paper follows a series of earlier works \cite{phillips_zeta-function_1933, rankin_van_1955, graham_algorithm_1986, petermann_divisor_1988, Lelechenko_linear_2014} where the optimal exponent pair was computed for certain restricted classes of exponent pairs and objective functions. As an example, we apply our methodology to obtain modest improvements on several open problems in analytic number theory. 

\subsection{Exponent pairs}
In this section we review the necessary background. For an overview of this subject, we refer the reader to \cite{graham_van_1991}. Let $\textbf{F}(N, P, \sigma, y, c)$ denote the set of functions $f : I = (a, b] \subseteq [N, 2N] \to \mathbb{R}$, such that $f$ has $P$ continuous derivatives and for which
\begin{equation}\label{exp_pair_hypothesis}
\left|f^{(p + 1)}(x) - \frac{\text{d}^p}{\text{d}x^p}(yx^{-\sigma})\right| \le c\left|\frac{\text{d}^p}{\text{d}x^p}(yx^{-\sigma})\right|,
\end{equation}
for all $x \in I$, $0 \le p \le P - 1$. Informally, this condition implies $f(x) = T F(x/N)$ where $T = yN^{1 - \sigma}$ is the order of $f$, and $F(u) \approx u^{-\sigma}$ ($1 \le u \le 2$) is a ``monomial" function in the sense of \cite[Ch.\ 3]{huxley_area_1996}. Next, let $(k, \ell)\in\mathbb{R}^2$ belong to the set\footnote{The condition $k + \ell < 1$ is not included in some treatments, however its presence does not matter in practice since all ``non-trivial" exponent pairs satisfy this inequality.}
\[
\{ 0 \le k \le 1/2 \le \ell \le 1, k + \ell < 1\} \cup \{ (\tfrac{1}{2}, \tfrac{1}{2}), (0, 1) \}.
\]
Then, $(k, \ell)$ is a (one-dimensional) exponent pair if for all $\sigma > 0$, there exists some $P = P(k, \ell, \sigma)$ and $c = c(k, \ell, \sigma) < 1/2$ such that 
\begin{equation}\label{exp_pair_defn}
\sum_{n \in I}e(f(n)) \ll_{k, \ell, \sigma} \left(\frac{y}{N^\sigma}\right)^kN^{\ell},\qquad (y \ge N^\sigma),
\end{equation}
uniformly for all $f \in \textbf{F}(N, P, \sigma, y, c)$. The exponent pair conjecture asserts that $(\varepsilon, 1/2 + \varepsilon)$ is an exponent pair for any $\varepsilon > 0$. Ignoring $\varepsilon$'s, the exponent pair conjecture is akin to obtaining ``square-root" cancellation in a large family of exponential sums. Among many other consequences, the conjecture implies at once the Lindel\"of hypothesis and solves (up to $\varepsilon$) the Dirichlet divisor problem. While this conjecture appears out of reach of current methods, many exponent pairs have been discovered, which we review next. 

\subsection{The set of known exponent pairs}\label{sec:known_exp_pairs} 
In this section we attempt a complete survey of known exponent pairs. Due to the proliferation of research in this area, any such attempt is primed to fail. Nevertheless, we believe such an endeavour is worthwhile since in many applications, sub-optimal exponent pairs are chosen. Researchers are often aware of this, remarking that some further improvement is possible by better choice of exponent pair. 

First we review exponent pairs that cannot be derived from other exponent pairs using convexity arguments or transforms, such as van der Corput iteration \cite{corput_1921}. From the triangle inequality, 
\[
\sum_{n \in I}e(f(n)) \ll N
\]
and hence $(0, 1)$ is an exponent pair (known also as the trivial exponent pair). The Bombieri--Iwaniec method \cite{bombieri_order_1986, bombieri_some_1986} has been used to find a number of exponent pairs of the form 
\begin{equation}\label{BI_exp_pair}
\left(\theta + \varepsilon, \frac{1}{2} + \theta + \varepsilon\right)
\end{equation}
for any $\varepsilon > 0$. The value of $\theta$ has been successively refined to 
\[
\theta = \frac{9}{56},\quad \frac{89}{560},\quad \frac{17}{108},\quad \frac{89}{570},\quad \frac{32}{205}, \quad \frac{13}{84},
\]
by Huxley--Watt \cite{huxley_exponential_1988}, Watt \cite{watt_exponential_1989}, Huxley--Kolesnik \cite{huxley_exponential_1991}, Huxley \cite{huxley_exponential_1993}, Huxley \cite{huxley_exponential_2005} and Bourgain \cite{bourgain_decoupling_2016} respectively. The method was also be used to obtain exponent pairs other than \eqref{BI_exp_pair}, such as 
\begin{equation}\label{huxley_pair_1}
\begin{split}
&\left(\frac{2}{13} + \varepsilon, \frac{35}{52} + \varepsilon\right),\quad \left(\frac{516247}{6629696} + \varepsilon, \frac{5080955}{6629696} + \varepsilon\right),\quad\left(\frac{6299}{43860} + \varepsilon, \frac{29507}{43860} + \varepsilon\right),\\
&\quad \left(\frac{771}{8116} + \varepsilon, \frac{1499}{2029} + \varepsilon\right),\quad \left(\frac{21}{232} + \varepsilon, \frac{173}{232} + \varepsilon\right),\quad \left(\frac{1959}{21656} + \varepsilon, \frac{16135}{21656} + \varepsilon\right),
\end{split}
\end{equation}
the first three of which are by Huxley--Watt \cite{huxley_hardy_1990}, Huxley--Kolesnik \cite{huxley_exponential_2001} (see also \cite{robert_fourth_2002}) and Huxley \cite[Ch.\ 17]{huxley_area_1996} respectively, and the last three by Sargos \cite{sargos_points_1995}. Huxley \cite[Table 17.3]{huxley_area_1996} also showed
\begin{equation}\label{huxley_pair_2}
\left( \frac{169}{1424\cdot 2^m - 338} + \varepsilon,1 - \frac{169}{1424\cdot 2^{m} - 338}\frac{712m + 1577}{712} + \varepsilon\right)
\end{equation}
is an exponent pair for any $m \ge 1$. Furthermore, by combining the results of \cite{sargos_points_1995}, \cite{huxley_area_1996} and  \cite{bourgain_decoupling_2016} we show
\begin{lemma}\label{new_exponent_pairs_lem}
The following are exponent pairs for any $\varepsilon > 0$: 
\begin{equation}\label{new_exp_pairs_4}
    \begin{split}
&\quad \left(\frac{4742}{38463} + \varepsilon, \frac{35731}{51284} + \varepsilon\right),\quad \left(\frac{18}{199} + \varepsilon, \frac{593}{796} + \varepsilon\right),\\
&\left(\frac{2779}{38033} + \varepsilon, \frac{58699}{76066} + \varepsilon\right),\quad\left(\frac{715}{10238} + \varepsilon, \frac{7955}{10238} + \varepsilon\right).
    \end{split}
\end{equation}
\end{lemma}

Many of the above results apply to a broader class of phase functions $f(x)$ than required for the definition of exponent pairs in \eqref{exp_pair_hypothesis}. Another type of result, assuming even weaker hypotheses, are the $m$th derivative tests, which imply exponent pairs of the form
\begin{equation}\label{kth_deriv_test_exp_pair}
\left(\vartheta + \varepsilon, 1 - (m - 1)\vartheta + \varepsilon\right),
\end{equation}
for some integer $m \ge 3$, $0 < \vartheta < 1/2$ and any $\varepsilon > 0$. Instead of requiring control of the first $P$ derivatives of a function, such results typically only require control of the $m$th derivative. Table \ref{exppair_table} shows some admissible choices of $m$ and $\vartheta$. 

\begin{comment}
    We remark that in many cases the exponent pair is only proven for restricted choices of $N$, however it is possible to complete the argument by appealing to other exponent pairs over missing ranges. For example, \cite[Thm.\ 1]{robert_fourth_2016} proved that one may take $(k, \ell) = (1/12 + \varepsilon, 3/4 + \varepsilon)$ in \eqref{exp_pair_defn}, provided that $N \ll T^{1/3}$, however for $N \gg T^{1/3}$ the desired result holds regardless thanks to the classical exponent pair $(1/9, 13/18)$.\footnote{In van der Corput notation, $(1/9, 13/18) = ABA^2B(0, 1)$.}
\end{comment}

\begin{table}[h]
\def\arraystretch{1.3}
\centering
\caption{Exponent pairs of the form \eqref{kth_deriv_test_exp_pair}.}
\begin{tabular}{|c|c|}
\hline
$m = 4$ & $\vartheta = 1/13$ \cite{robert_fourth_2002}\\
\hline
$m = 8$ & $\vartheta = 1/204$ \cite[Thm.\ 3]{sargos_analog_2003}\\
\hline
$m = 9$ & $\vartheta = 7/2640$ \cite[Thm.\ 4]{sargos_analog_2003}, $\vartheta = 1/360$ \cite{robert_analogue_2002}\\
\hline
$m = 10$ & \begin{tabular}{@{}c@{}}
$\vartheta = 1/716$ \cite{sargos_analog_2003}, $\vartheta = 1/649$ \cite{robert_applications_2001}, \\$\vartheta = 7/4540$ \cite{robert_analogue_2002}, $\vartheta = 1/615$ \cite{robert_quelques_2002}
\end{tabular}\\
\hline
$m = 11$ & $\vartheta = 1/915$ \cite{robert_quelques_2002}\\
\hline
\end{tabular}
\label{exppair_table}
\end{table}

Yet another class of exponent pairs may be derived from estimates of Vinogradov's mean value integral. Heath-Brown \cite[(6.17.4)]{titchmarsh_theory_1986} showed that 
\begin{equation}\label{hb_exponent_pair1}
(a_m, b_m) = \left(\frac{1}{25m^2(m - 2)\log m}, 1 - \frac{1}{25m^2\log m}\right)
\end{equation}
is an exponent pair for all $m \ge 3$. Incorporating the essentially optimal estimates of Vinogradov's integral in 
 \cite{wooley_cubic_2016, bourgain_proof_2016}, Heath-Brown \cite{heathbrown_new_2017} showed that 
\begin{equation}\label{hb_exponent_pair}
(p_m, q_m) = \left(\frac{2}{(m - 1)^2(m + 2)},1 - \frac{3m - 2}{m(m - 1)(m + 2)} + \varepsilon\right)
\end{equation}
is an exponent pair for all integers $m \ge 3$ and any $\varepsilon > 0$.

All other known exponent pairs can be generated from the above primitive exponent pairs using convexity and transformations of existing exponent pairs. Rankin \cite{rankin_van_1955} first observed that the set of exponent pairs is convex --- that is, if $(k_1, \ell_1)$ and $(k_2, \ell_2)$ are exponent pairs, then so is 
\[
\left(\lambda k_1 + (1 - \lambda)k_2, \lambda \ell_1 + (1 - \lambda)\ell_2\right),
\]
for any $\lambda \in [0, 1]$. The van der Corput method can be used to generate more exponent pairs by transforming an exponential sum into simpler sums via two processes, known as $A$ and $B$, both of which have versions in any number of dimensions. Here we review the one-dimensional case only. The $A$ process, also known as Weyl-differencing, expresses an exponential sum in terms of another exponential sum whose phase function is easier to control. By applying the $A$ process, if $(k, \ell)$ is an exponent pair, then so is 
\begin{equation}\label{A_process_defn}
A(k, \ell) := \left(\frac{k}{2k + 2}, \frac{\ell}{2k + 2} + \frac{1}{2}\right).
\end{equation}
The $B$ process, also known as Poisson summation, expresses an exponential sum in terms of another exponential sum that is typically shorter. By applying the $B$ process, if $(k, \ell)$ is an exponent pair, then so is
\begin{equation}\label{B_process_defn}
B(k, \ell) := \left(\ell - \frac{1}{2}, k + \frac{1}{2}\right).
\end{equation}
Other types of transformations are also known. For example, Sargos \cite[Thm.\ 5]{sargos_analog_2003} showed that if $(k, \ell + \varepsilon)$ is an exponent pair (for sufficiently small $\varepsilon > 0$), then so is 
\begin{equation}\label{C_process_defn}
C(k, \ell + \varepsilon) := \left(\frac{k}{12(1 + 4k)}, \frac{11(1 + 4k) + \ell}{12(1 + 4k)} + \varepsilon\right).
\end{equation}

\subsection{Finding optimal exponent pairs} The task of finding the optimal exponent pair for a given problem is highly non-trivial. Some partial results are known for special objective functions and for certain subsets of known exponent pairs. Rankin \cite{rankin_van_1955} built on the work of Phillips \cite{phillips_zeta-function_1933} to find the best exponent pair obtainable from $(0, 1)$ and van der Corput iteration, for the objective function $F(k, \ell) = k + \ell$. This objective function was interesting due to its application of bounding the Riemann zeta-function on the critical line. Later, Graham \cite{graham_algorithm_1986} developed an algorithm for objective functions of the form 
\begin{equation}\label{graham_objfunc}
F(k, \ell) = \frac{ak + b\ell + c}{dk + e\ell + f}.
\end{equation}
Petermann \cite{petermann_divisor_1988} expanded this analysis to also include the Bombieri--Iwaniec exponent pair $(9/56 + \varepsilon, 37/56 + \varepsilon)$. This was also the first analysis to consider exponent pairs of the form \eqref{BI_exp_pair}. More recently, Lelechenko \cite{Lelechenko_linear_2014} provided an analysis for objective functions of the form $\max\{F_1(k, \ell), \ldots, F_n(k, \ell)\}$, where each $F_i$ is of the form \eqref{graham_objfunc}. The class of exponent pairs was also expanded to include \eqref{BI_exp_pair} for various $\theta \ge 32/205$. Very recently, Cassaigne, Drappeau and Ramar\'{e} \cite{cassaigne_notes_2023} studied the geometry of the set of exponent pairs generated by \eqref{BI_exp_pair} with $\theta \ge 89/560$ and van der Corput iteration. 

Notably missing from existing analysis are considerations for:
\begin{enumerate}
    \item exponent pairs of the form \eqref{huxley_pair_1},  \eqref{huxley_pair_2}, \eqref{kth_deriv_test_exp_pair}, \eqref{hb_exponent_pair1} and \eqref{hb_exponent_pair}, 
    \item exponent pairs that can be obtained from process $C$, defined in \eqref{C_process_defn},
    \item exponent pairs that can be obtained from convexity arguments,
    \item general objective functions $F(k, \ell)$. 
\end{enumerate}
The goal of this work is to characterise completely the set of currently-known one-dimensional exponent pairs, and to provide a method to consider arbitrary objective functions $F$. In particular, we show that all known exponent pairs lie in a certain convex hull. The task of finding favourable exponent pairs thus reduces to a constrained optimisation problem which can be easily solved using a numerical optimisation package. We then demonstrate this approach on a range of open problems. 

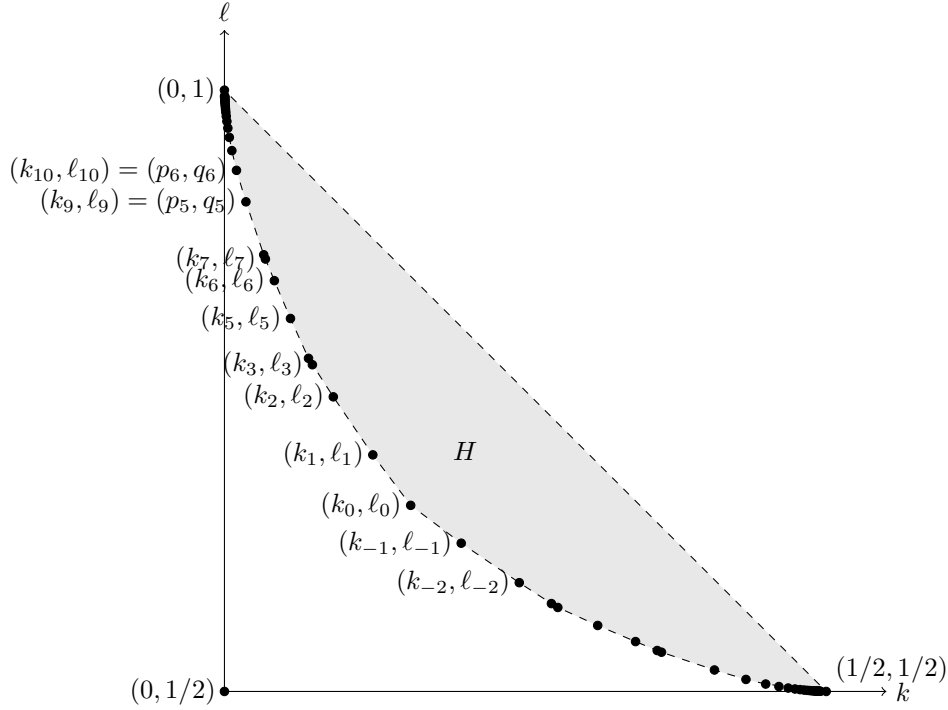
\begin{figure}
\centering
\begin{tikzpicture}[scale=16]
% axis
\draw[->, thin] (0, 0) -> (0, 0.55) node[anchor=south]{$\ell$};
\draw[->, thin] (0, 0) -> (0.55, 0) node[anchor=west]{$k$};

% region
\draw[black, dashed, fill={rgb:black,1;white,10}] 
(85/189, 1/162) -- 
(13/30, 1/100) -- 
(57/140, 1/56) -- 
(0.3631425, 0.0326394) -- 
(0.3595694, 0.0340463) -- 
(0.3415898, 0.0414746) -- 
(0.3101290, 0.0548779) -- 
(0.2770072, 0.0698378) -- 
(0.2716851, 0.0730681) -- 
(0.2449748, 0.0904522) -- 
(0.1967280, 0.1232873) -- 
(13/84 , 13/84) -- 
(0.1232873, 0.1967280) -- 
(0.0904522, 0.2449748) -- 
(0.0698378, 0.2770072) -- 
(0.0548779, 0.3101290) -- 
(0.0414746, 0.3415898) -- 
(0.0340463, 0.3595694) -- 
(0.0326394, 0.3631425) -- 
(1/56, 57/140) --
(1/100, 13/30) --
(1/162, 85/189) --
(1/245, 129/280) --
(1/352, 371/792) --
(1/486, 64/135) --
(0, 1/2) -- (1/2, 0) -- cycle;

% labels 
\filldraw[black] (129/280, 1/245) circle (0.1pt);
\filldraw[black] (371/792, 1/352) circle (0.1pt);
\filldraw[black] (64/135, 1/486) circle (0.1pt);
\filldraw[black] (342/715, 1/650) circle (0.1pt);
\filldraw[black] (445/924, 1/847) circle (0.1pt);
\filldraw[black] (1133/2340, 1/1080) circle (0.1pt);
\filldraw[black] (177/364, 1/1352) circle (0.1pt);
\filldraw[black] (871/1785, 1/1666) circle (0.1pt);
\filldraw[black] (1057/2160, 1/2025) circle (0.1pt);
\filldraw[black] (0.4905185758513932,  0.00041118421052631577) circle (0.1pt);
\filldraw[black] (0.4915032679738562,  0.00034602076124567473) circle (0.1pt);
\filldraw[black] (0.4923419660261765,  0.00029394473838918284) circle (0.1pt);
\filldraw[black] (0.4930622009569378, 0.0002518257365902795) circle (0.1pt);
\filldraw[black] (0.4936853002070393, 0.0002173913043478261) circle (0.1pt);
\filldraw[black] (0.49422799422799424, 0.0001889644746787604) circle (0.1pt);
\filldraw[black] (0.49470355731225296, 0.00016528925619834712) circle (0.1pt);
\filldraw[black] (0.4951226309921962, 0.00014541224371092046) circle (0.1pt);

\filldraw[black] (85/189, 1/162) circle (0.1pt);
\filldraw[black] (13/30, 1/100) circle (0.1pt);
\filldraw[black] (57/140, 1/56) circle (0.1pt);
\filldraw[black] (0.3631425, 0.0326394) circle (0.1pt);
\filldraw[black] (0.3595694, 0.0340463) circle (0.1pt);
\filldraw[black] (0.3415898, 0.0414746) circle (0.1pt);
\filldraw[black] (0.3101290, 0.0548779) circle (0.1pt);
\filldraw[black] (0.2770072, 0.0698378) circle (0.1pt);
\filldraw[black] (0.2716851, 0.0730681) circle (0.1pt);
\filldraw[black] (0.2449748, 0.0904522) circle (0.1pt) node[anchor=east]{$(k_{-2}, \ell_{-2})$};
\filldraw[black] (0.1967280, 0.1232873) circle (0.1pt) node[anchor=east]{$(k_{-1}, \ell_{-1})$};

\filldraw[black] (13/84, 13/84) circle (0.1pt) node[anchor=east]{$(k_0, \ell_0)$};

\filldraw[black] (0.1232873, 0.1967280) circle (0.1pt) node[anchor=east]{$(k_1, \ell_1)$};
\filldraw[black] (0.0904522, 0.2449748) circle (0.1pt) node[anchor=east]{$(k_2, \ell_2)$};
\filldraw[black] (0.0730681, 0.2716851) circle (0.1pt) node[anchor=east]{$(k_3, \ell_3)$};
\filldraw[black] (0.0698378, 0.2770072) circle (0.1pt);
\filldraw[black] (0.0548779, 0.3101290) circle (0.1pt) node[anchor=east]{$(k_5, \ell_5)$};
\filldraw[black] (0.0414746, 0.3415898) circle (0.1pt) node[anchor=east]{$(k_6, \ell_6)$};
\filldraw[black] (0.0340463, 0.3595694) circle (0.1pt) node[anchor=east]{$(k_7, \ell_7)$};
\filldraw[black] (0.0326394, 0.3631425) circle (0.1pt);
\filldraw[black] (1/56, 57/140) circle (0.1pt) node[anchor=east]{$(k_9, \ell_9) = (p_5, q_5)$};
\filldraw[black] (1/100, 13/30) circle (0.1pt) node[anchor=east]{$(k_{10}, \ell_{10}) = (p_6, q_6)$};
\filldraw[black] (1/162, 85/189) circle (0.1pt);
\filldraw[black] (1/245, 129/280) circle (0.1pt);
\filldraw[black] (1/352, 371/792) circle (0.1pt);
\filldraw[black] (1/486, 64/135) circle (0.1pt);
\filldraw[black] (1/650, 342/715) circle (0.1pt);
\filldraw[black] (1/847, 445/924) circle (0.1pt);
\filldraw[black] (1/1080, 1133/2340) circle (0.1pt);
\filldraw[black] (1/1352, 177/364) circle (0.1pt);
\filldraw[black] (1/1666, 871/1785) circle (0.1pt);
\filldraw[black] (1/2025, 1057/2160) circle (0.1pt);
\filldraw[black] ( 0.00041118421052631577 , 0.4905185758513932 ) circle (0.1pt);
\filldraw[black] ( 0.00034602076124567473 , 0.4915032679738562 ) circle (0.1pt);
\filldraw[black] ( 0.00029394473838918284 , 0.4923419660261765 ) circle (0.1pt);
\filldraw[black] ( 0.0002518257365902795 , 0.4930622009569378 ) circle (0.1pt);
\filldraw[black] ( 0.0002173913043478261 , 0.4936853002070393 ) circle (0.1pt);
\filldraw[black] ( 0.0001889644746787604 , 0.49422799422799424 ) circle (0.1pt);
\filldraw[black] ( 0.00016528925619834712 , 0.49470355731225296 ) circle (0.1pt);
\filldraw[black] ( 0.00014541224371092046 , 0.4951226309921962 ) circle (0.1pt);

\filldraw[black] (0, 1/2) circle (0.1pt) node[anchor=east]{$(0, 1)$};
\filldraw[black] (1/2, 0) circle (0.1pt) node[anchor=south west]{$(1/2, 1/2)$};
\filldraw[black] (0, 0) circle (0.1pt) node[anchor=east]{$(0, 1/2)$};

\node at (0.2, 0.2) {$H$};
\end{tikzpicture}
\caption{Plot of $H$}
\label{H_theta_plot}
\end{figure}

\subsection{The convex hull $H$}\label{convex_hull_sec}
In this section we explicitly compute the convex hull containing all exponent pairs reviewed in the previous section. For all integers $n$, define the point $(k_n, \ell_n)$ as
\begin{align}
(k_0, \ell_0) &:= \left(\frac{13}{84}, \frac{55}{84}\right),\quad (k_1, \ell_1) := \left(\frac{4742}{38463}, \frac{35731}{51284}\right),\quad (k_2, \ell_2) := \left(\frac{18}{199}, \frac{593}{796}\right),\notag\\ 
&(k_3, \ell_3) := \left(\frac{2779}{38033}, \frac{58699}{76066}\right),\quad (k_4, \ell_4) := \left(\frac{715}{10238}, \frac{7955}{10238}\right)\label{knln_defn}
\end{align}
and 
\[
(k_n, \ell_n) := \begin{cases}
    A(k_{n - 4}, \ell_{n - 4}),&5 \le n \le 8,\\
    (p_{n - 4}, q_{n - 4}),&n \ge 9,\\
    B(k_{-n}, \ell_{-n}),&n < 0.
\end{cases}
\]
Ignoring $\varepsilon$'s, $(k_0, \ell_0)$ comes from Bourgain's \cite{bourgain_decoupling_2016} exponent pair, $(k_1, \ell_1), \ldots, (k_4, \ell_4)$ come from Lemma \ref{new_exponent_pairs_lem} and $(p_{m}, q_m)$ comes from  the series of exponent pairs defined in \eqref{hb_exponent_pair}.

Note that this set of points is symmetric about the line $\ell = k + 1/2$, and that 
\[
\lim_{n \to \infty}(k_n, \ell_n) = (0, 1),\qquad \lim_{n \to -\infty}(k_n, \ell_n) = (1/2, 1/2).
\]
Then, we define $H$ in the following way. 
\begin{definition}[The convex hull $H$]\label{convex_hull_H_defn}
Let $H \subset \mathbb{R}^2$ be the set of points enclosed by the (infinite edged) polygon formed by line segments joining $(k_n, \ell_n)$ and $(k_{n + 1}, \ell_{n + 1})$ for all integers $n$, combined with the line segment joining $(0, 1)$ and $(1/2, 1/2)$. 
\end{definition}

Figure \ref{H_theta_plot} shows a graph of $H$. Note that the vertices $(k_n, \ell_n)$ for $|n| \ge 9 $ correspond to elements of the series \eqref{hb_exponent_pair}. Indeed \eqref{hb_exponent_pair} form the best-known exponent pairs close to the extremities $(0, 1)$ and $(1/2, 1/2)$. Figure \ref{hb_exp_pair_plot} records a comparison of \eqref{hb_exponent_pair} and a series of exponent pairs of the form $A^p(\frac{13}{84} + \varepsilon, \frac{55}{84} + \varepsilon)$.

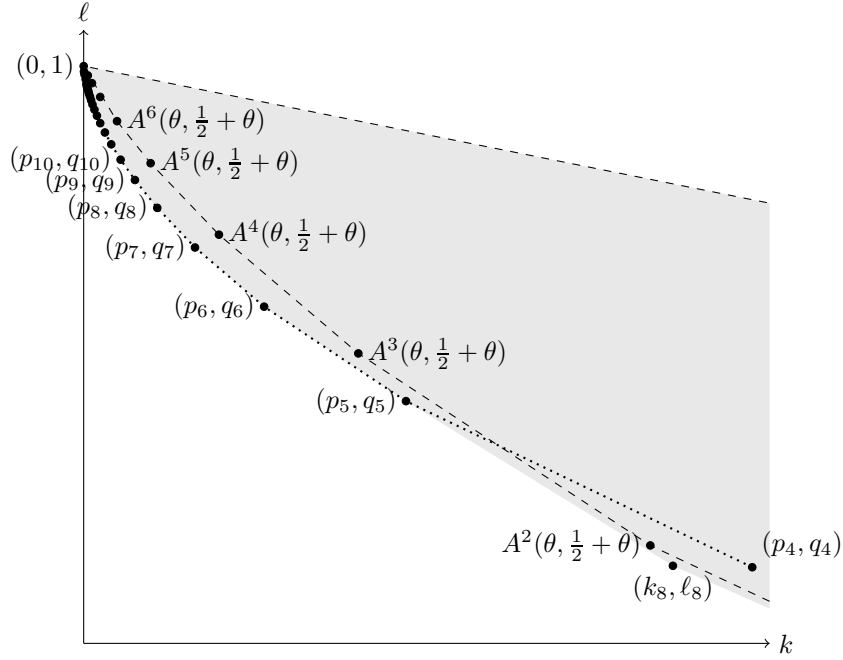
\begin{figure}
\centering
\begin{tikzpicture}[x=240cm,y=48cm]
% axis
\draw[->, thin] (0, 0.34) -> (0, 0.51) node[anchor=south]{$\ell$};
\draw[->, thin] (0, 0.34) -> (0.038, 0.34) node[anchor=west]{$k$};

\fill[fill={rgb:black,1;white,10}] 
( 0.038 , 0.351649 - 0.002) -- % exaggerated to show that the hull extends beyond
(0.0326394, 0.3615) -- % exaggerated to show that the hull extends beyond A^2(S)
(1/56, 57/140) -- 
(1/100, 13/30) -- 
(1/162, 85/189) -- 
(1/245, 129/280) -- 
(1/352, 371/792) -- 
(1/486, 64/135) -- 
(1/650, 342/715) -- 
(1/847, 445/924) -- 
(1/1080, 1133/2340) -- 
(1/1352, 177/364) -- 
(1/1666, 871/1785) -- 
(1/2025, 1057/2160) -- 
( 0.00041118421052631577 , 0.4905185758513932 ) -- 
( 0.00002 , 0.4985 ) -- 
(0, 1/2) -- (0.038, 1/2 - 0.038) -- cycle;

% region
\draw[black, dashed] 
( 0.038 , 0.351649 ) --
( 13/414 , 76/207 ) -- 
( 13/854 , 359/854 ) -- 
( 13/1734 , 131/289 ) -- 
( 13/3494 , 1653/3494 ) -- 
( 13/7014 , 1700/3507 ) -- 
(0, 1/2) -- (0.038, 1/2 - 0.038);

% labels 
\filldraw[black] (13/414 , 76/207) circle (0.05cm) node[anchor=east]{$A^2(\theta, \frac{1}{2} + \theta)$};
\filldraw[black] (13/854 , 359/854) circle (0.05cm) node[anchor=west]{$A^3(\theta, \frac{1}{2} + \theta)$};
\filldraw[black] (13/1734 , 131/289) circle (0.05cm) node[anchor=west]{$A^4(\theta, \frac{1}{2} + \theta)$};
\filldraw[black] (13/3494 , 1653/3494) circle (0.05cm) node[anchor=west]{$A^5(\theta, \frac{1}{2} + \theta)$};
\filldraw[black] (13/7014 , 1700/3507) circle (0.05cm) node[anchor=west]{$A^6(\theta, \frac{1}{2} + \theta)$};
\filldraw[black] (13/14054, 6907/14054) circle (0.05cm);
\filldraw[black] (0.000462074358, 6967/14067) circle (0.05cm);
\filldraw[black] (0.000230930472, 0.497406473158) circle (0.05cm); 
\filldraw[black] (0, 1/2) circle (0.05cm) node[anchor=east]{$(0, 1)$};

\draw[black, dotted, thick]
(1/27, 13/36) --
(1/56, 57/140) --
(1/100, 13/30) --
(1/162, 85/189) --
(1/245, 129/280) --
(1/352, 371/792) --
(1/486, 64/135) --
(1/650, 342/715) --
(1/847, 445/924) --
(1/1080, 1133/2340) --
(1/1352, 177/364) --
(1/1666, 871/1785) --
(1/2025, 1057/2160) --
(0, 1/2);

\filldraw[black] (0.0326394, 0.3615) circle (0.05cm) node[anchor=north]{$(k_8, \ell_8)$};
\filldraw[black] (1/27, 13/36) circle (0.05cm) node[anchor=south west]{$(p_4, q_4)$};
\filldraw[black] (1/56, 57/140) circle (0.05cm) node[anchor=east]{$(p_5, q_5)$};
\filldraw[black] (1/100, 13/30) circle (0.05cm) node[anchor=east]{$(p_6, q_6)$};
\filldraw[black] (1/162, 85/189) circle (0.05cm) node[anchor=east]{$(p_7, q_7)$};
\filldraw[black] (1/245, 129/280) circle (0.05cm) node[anchor=east]{$(p_8, q_8)$};
\filldraw[black] (1/352, 371/792) circle (0.05cm) node[anchor=east]{$(p_9, q_9)$};
\filldraw[black] (1/486, 64/135) circle (0.05cm) node[anchor=east]{$(p_{10}, q_{10})$};
\filldraw[black] (1/650, 342/715) circle (0.05cm);
\filldraw[black] (1/847, 445/924) circle (0.05cm);
\filldraw[black] (1/1080, 1133/2340) circle (0.05cm);
\filldraw[black] (1/1352, 177/364) circle (0.05cm);
\filldraw[black] (1/1666, 871/1785) circle (0.05cm);
\filldraw[black] (1/2025, 1057/2160) circle (0.05cm);
\filldraw[black] ( 0.00041118421052631577 , 0.4905185758513932 ) circle (0.05cm);
\filldraw[black] ( 0.00034602076124567473 , 0.4915032679738562 ) circle (0.05cm);
\filldraw[black] ( 0.00029394473838918284 , 0.4923419660261765 ) circle (0.05cm);
\filldraw[black] ( 0.0002518257365902795 , 0.4930622009569378 ) circle (0.05cm);
\filldraw[black] ( 0.0002173913043478261 , 0.4936853002070393 ) circle (0.05cm);
\filldraw[black] ( 0.0001889644746787604 , 0.49422799422799424 ) circle (0.05cm);
\filldraw[black] ( 0.00016528925619834712 , 0.49470355731225296 ) circle (0.05cm);
\filldraw[black] ( 0.00014541224371092046 , 0.4951226309921962 ) circle (0.05cm);

\filldraw[black] ( 0.0001 , 0.4963844797178131 ) circle (0.05cm);
\filldraw[black] ( 0.00005 , 0.49766068589598 ) circle (0.05cm);
\filldraw[black] ( 0.00002 , 0.4985 ) circle (0.05cm);

\end{tikzpicture}
\caption{Plot of the series of exponent pairs of the form $(p_m, q_m)$ and $A^p(\theta, \frac{1}{2} + \theta)$ with $\theta = 13/84$ respectively. $H$ is shaded in grey.}
\label{hb_exp_pair_plot}
\end{figure}

We record the following properties of $H$: 
\begin{enumerate}
    \item $H$ is convex (proved in Lemma \ref{h_theta_convex_lem} below).
    \item $H$ is symmetric about $\ell = k + 1/2$, and hence closed under the $B$ transformation (defined in \eqref{B_process_defn}). That is, if $(k, \ell) \in H$ then $B(k, \ell) \in H$. 
    \item $H$ is closed under both $A$ and $C$ transformations, defined in \eqref{A_process_defn} and \eqref{C_process_defn} respectively (proved in Lemma \ref{H_closed_A_lem} below). 
    \item $H$ contains all exponent pairs reviewed in \S \ref{sec:known_exp_pairs} (proved in Lemma \ref{H_contains_lem} below). 
\end{enumerate}
In light of the above observations, we conclude that 
\begin{theorem}\label{reg_thm_1}
If $I$ is the interior of $H$, then 
\[
I \cup \{(0, 1), (\tfrac{1}{2}, \tfrac{1}{2})\}
\]
is the set of all known one-dimensional exponent pairs.
\end{theorem}

In particular, $I$ contains
\begin{enumerate}
    \item all classical exponent pairs formed by van der Corput iteration and the trivial pair, such as
    \[
    (\tfrac{1}{6}, \tfrac{2}{3}) = AB(0, 1),\qquad (\tfrac{11}{82}, \tfrac{57}{82}) = ABA^3B(0, 1),
    \]
    \item all transformations of sporadic exponent pairs of the form \eqref{BI_exp_pair}, such as 
    \[
    \left(\tfrac{13}{414} + \varepsilon, \tfrac{359}{414} + \varepsilon\right) = A^2\left(\tfrac{13}{84} + \varepsilon, \tfrac{55}{84} + \varepsilon\right),
    \]
    \item all known exponent pairs of the form \eqref{huxley_pair_2} --- \eqref{hb_exponent_pair}, such as
    \[
    \left(\tfrac{1}{27}, \tfrac{13}{36} + \varepsilon\right),\qquad \left(\tfrac{1}{56}, \tfrac{57}{140} + \varepsilon\right),\qquad\left(\tfrac{1}{915}, \tfrac{181}{183}\right),
    \]
    \item all exponent pairs that can be formed by convexity, such as
    \[
    \left(\tfrac{3571}{16296} + \varepsilon', \tfrac{9955}{16296} + \varepsilon'\right) = \tfrac{1}{2}\left(\tfrac{13}{84} + \varepsilon, \tfrac{55}{84} + \varepsilon\right) + \tfrac{1}{2}BA\left(\tfrac{13}{84} + \varepsilon, \tfrac{55}{84} + \varepsilon\right)
    \]
    for any $\varepsilon, \varepsilon' > 0$. 
\end{enumerate}
As new exponent pairs are discovered, the definition of $H$ will also need to be updated. The availability of efficient algorithms for computing convex hulls makes this process straightforward. 

\section{Applications}\label{sec:applications}

In this section we show that simply by choosing the best exponent pair contained inside $H$, we are able to make progress on several open problems. In some problems we are able to solve analytically for the optimal exponent pair. In others, we are content with candidate optima found using a numerical constrained optimisation program. 

To perform the numerical optimisation, in practice we approximate $H$ by a convex polygon with finitely many vertices, given by 
\begin{equation}\label{HN_vertices}
\{(k_n, \ell_n)\}_{|n| \le N} \cup \{(0, 1), (1/2, 1/2)\},
\end{equation}
for some large $N$ (say 10000). This polygon closely approximates $H$ even with moderately large $N$, and furthermore is strictly contained inside $H$, so that any solution found is guaranteed to be valid (even if it may be slightly sub-optimal). In all the applications below that make use of numerical optimisation, the found solution was far away from the extremeties $(0, 1)$ and $(1/2, 1/2)$, and no improvement was obtained by increasing $N$. 

Many of the results of this section are stated in terms of a parameter that varies continuously over some interval. Examples include $A \in [866/65, \infty)$ in Theorem \ref{M_bound_larger_A} and $\sigma \in [1/2, 1]$ in Theorem \ref{mu_est_thm} and \ref{zeta_bound_thm}. As a general remark we often find that as such a parameter varies, the optimal exponent pair for that parameter choice traverses along the boundary of $H$. There are two distinct regimes -- when the optimal exponent pair is close to extremities $(0, 1)$ and $(1/2, 1/2)$, and when it is close to the ``centre line" $\ell = k + 1/2$. In the former case, such as Theorem \ref{M_bound_12_thm} and \ref{zeta_bound_thm}, we tend to use exponent pairs of the form \eqref{hb_exponent_pair}, while in the latter case, such as Theorem \ref{M_bound_larger_A} and \ref{zeta_bound_thm}, we use exponent pairs of the form \eqref{BI_exp_pair} and \eqref{new_exp_pairs_4}, combined with van der Corput iteration.

\subsection{Moments of $\zeta(s)$}
Let $\zeta(s)$ denote the Riemann zeta-function. A central task in analytic number theory is to bound the moments 
\[
I_{2k} := \int_1^T|\zeta(1/2 + it)|^{2k}\text{d}t,
\]
which represents a mean-value result on the order of the zeta-function on the critical line. It is conjectured that $I_{2k} \ll_{\varepsilon} T^{1 + \varepsilon}$ for all $k > 0$, a result that is equivalent to the well-known Lindel\"of Hypothesis. Currently, the conjecture is only proven for $k = 1$ (e.g.\ \cite[Thm.~2.41]{hardy_contributions_1916}) and $k = 2$ \cite[Thm.~D]{hardy_approximate_1923} (in fact, precise asymptotics are known for $I_{2k}$ in these cases \cite{hardy_contributions_1916, ingham_mean_1928}). For higher $k$ some partial results are known, which are sometimes of independent interest due to their implications for zero-density estimates \cite{ivic_riemann_2003}. For instance, it is known that 
\[
I_A \ll_{\varepsilon} T^{M(A) + \varepsilon},\qquad M(A) \le \begin{cases}
1 + (A - 4)/8, &4\le A < 12,\\
2 + 3(A - 12)/22, &12 \le A < 178/13,\\
1 + 35(A - 6)/216, &178/13 \le A < 14,\\
1 + 9(A - 6)/56, &14 \le A.
\end{cases}
\]
This result combines contributions from multiple works, including \cite{hardy_approximate_1923, heathbrown_fourth_1979, ivic_riemann_2003, ivic_dirichlet_1989}. More recently Ivi\'{c} \cite{ivic_riemann_2005} has shown that $M(1238/75) \le 601/225$. Using an older result due to Ivi\'{c} \cite[Thm.\ 8.2]{ivic_riemann_2003}, combined with exponent pairs of the form \eqref{hb_exponent_pair}, we also obtain new estimates of $M(A)$ close to $A = 12$, as well as improved bounds on larger $A$. 
\begin{theorem}\label{M_bound_12_thm}
We have
\[
M(12 + \delta) \le 2 + \frac{\delta}{8} + \frac{3\sqrt{510}}{7568}\delta^{3/2},\qquad 0 < \delta \le \frac{86}{65}.
\] 
\end{theorem}
\begin{theorem}\label{M_bound_larger_A}
\begin{comment}
-6 (Fraction(13, 30), Fraction(51, 100)) (Fraction(85, 189), Fraction(41, 81)) 11*A/80 + 69/200 3302/255 866/65
-5 (Fraction(57, 140), Fraction(29, 56)) (Fraction(13, 30), Fraction(51, 100)) 8*A/57 + 35/114 866/65 14
-4 (Fraction(9, 26), Fraction(7, 13)) (Fraction(57, 140), Fraction(29, 56)) 97*A/684 + 49/171 14 16
-3 (Fraction(55, 194), Fraction(55, 97)) (Fraction(9, 26), Fraction(7, 13)) 55*A/378 + 43/189 16 1048/55
-2 (Fraction(1, 4), Fraction(7, 12)) (Fraction(55, 194), Fraction(55, 97)) 55*A/376 + 10/47 1048/55 64/3
-1 (Fraction(13, 84), Fraction(55, 84)) (Fraction(1, 4), Fraction(7, 12)) 37*A/244 + 6/61 64/3 440/13
0 (Fraction(1, 12), Fraction(3, 4)) (Fraction(13, 84), Fraction(55, 84)) 31*A/196 - 6/49 440/13 64
\end{comment}
We have
\[
M(A) \leq \begin{cases}
    (16A + 35)/114 ,& \frac{866}{65} \leq A < 14, \\
	 (176677A + 358428)/1246476 ,& 14 \leq A < \frac{122304}{7955} = 15.37\ldots, \\
	 (779A + 1398)/5422 ,& \frac{122304}{7955} \leq A < \frac{910020}{58699} = 15.50\ldots, \\
	 3(1661A + 2856)/34532 ,& \frac{910020}{58699} \leq A < \frac{9604}{593} = 16.19\ldots, \\
	 (405277A + 677194)/2800950 ,& \frac{9604}{593} \leq A < \frac{629068}{35731} = 17.60\ldots, \\
	 (40726597A + 64268678)/280113282 ,& \frac{629068}{35731} \leq A < \frac{13789}{709} = 19.44\ldots, \\
	 3(46A + 49)/926 ,& \frac{13789}{709} \leq A < \frac{204580}{10333} = 19.79\ldots,\\
	 (3475A + 3236)/23168 ,& \frac{204580}{10333} \leq A < \frac{4252}{195} = 21.80\ldots, \\
	 7(39945A + 33704)/1857036 ,& \frac{4252}{195} \leq A < \frac{812348}{30267} = 26.83\ldots, \\
	 (37A + 24)/244 ,& \frac{812348}{30267} \leq A < \frac{440}{13} = 33.84\ldots, \\
	 (31A - 24)/196 ,& \frac{440}{13} \leq A < \frac{203087}{4742} = 42.82\ldots, \\
	 7(31519A - 33704)/1385180 ,& \frac{203087}{4742} \leq A < \frac{3516129}{65729} = 53.49\ldots, \\
    1 + 13(A - 6)/84 ,& \frac{3516129}{65729} \leq A.
\end{cases}
\]
\end{theorem}
In particular, we have 
\begin{align*}
M(13) &\le 2.1340,&M(14) &\le 2.2720,&M(15) &\le 2.4137,\\
M(16) &\le 2.5570,&M(17) &\le 2.7016,&M(18) &\le 2.8466.
\end{align*}

Ivi\'{c} \cite{ivic_mean_2004} also studied another type of partial result --- hybrid moments of the form  
\[
I_{4 + 2j}(\sigma) = \int_1^T|\zeta(1/2 + it)|^4|\zeta(\sigma + it)|^{2j}\text{d}t
\]
and showed that $I_6(\sigma) \ll_{\varepsilon} T^{1 + \varepsilon}$ for $\sigma \ge 5/6$ and $I_{8}(\sigma) \ll_{\varepsilon} T^{1 + \varepsilon}$ for $\sigma \ge 1953/1984 = 0.9843\ldots$. Note that if the ranges of $\sigma$ can be extended to $\sigma \ge 1/2$, then the conjectured results of $I_{2k}$ for $k = 3$ and $4$ follow immediately. Ivi\'{c} and Zhai \cite{ivic_mean_2012} improved (amongst other results) that $I_8(\sigma) \ll_{\varepsilon} T^{1 + \varepsilon}$ to all $\sigma > 37/38 = 0.9736\ldots$. By applying the method of \S \ref{convex_hull_sec}, we obtain the following improvement. 

\begin{theorem}\label{zeta_moment_thm}
We have 
\[
\int_1^T|\zeta(1/2 + it)|^4|\zeta(\sigma + it)|^4\text{d}t \ll_{\varepsilon} T^{1 + \varepsilon},\qquad \sigma > \frac{309}{320} = 0.965625.
\]
\end{theorem}

\subsection{Bounds on $\zeta(s)$ in the critical strip}
For any $1/2 \le \sigma \le 1$, let $\mu(\sigma)$ be the infimum of numbers $\mu$ such that 
\[
\zeta(\sigma + it) \ll t^{\mu},\qquad (t \to \infty).
\]
The unproven Lindel\"of hypothesis is equivalent to the assertion that $\mu(\sigma) = 0$ for $1/2 \le \sigma \le 1$. Among many other applications, bounds on $\mu(\sigma)$ have been used to construct zero-free regions and zero-density estimates for $\zeta(s)$. Various bounds are known to hold for specific values of $\sigma$, such as
\[
\mu\left(\tfrac{1}{2}\right) \le \tfrac{13}{84},\quad \mu\left(\tfrac{1934}{3655}\right) \le \tfrac{6299}{43860},\quad \mu\left(\tfrac{3}{5}\right) \le \tfrac{1409}{12170},\quad \mu\left(\tfrac{11}{15}\right) \le \tfrac{1}{15}, \quad \mu\left(\tfrac{4}{5}\right) \le \tfrac{3}{71},
\] 
\[
\mu\left(\tfrac{49}{51}\right) \le \tfrac{1}{204},\quad \mu\left(\tfrac{361}{370}\right) \le \tfrac{1}{370}.
\]
In order of appearance, these results are due to Bourgain \cite{bourgain_decoupling_2016}, Huxley \cite[(21.2.4)]{huxley_area_1996}, Lelechenko \cite{Lelechenko_linear_2014}, Heath-Brown \cite{demeter_small_2020}, Lelechenko \cite{Lelechenko_linear_2014}, Sargos \cite{sargos_analog_2003} and Sargos \cite{sargos_analog_2003}. The bound on $\mu(1/2)$ in particular represents the best-known result after over a century of effort, see for example \cite{corput_1921, titchmarsh_van_1931, phillips_zeta-function_1933, titchmarsh_order_1942, min_on_1949, rankin_van_1955, haneke_verscharfung_1963, kolesnik_order_1982, bombieri_order_1986, bombieri_some_1986, watt_exponential_1989, huxley_exponential_1991, huxley_exponential_1993, huxley_area_1996, huxley_exponential_2005, bourgain_decoupling_2016}. Furthermore, the classical estimates of van der Corput and Hardy--Littlewood (see \S 5.15 in \cite{titchmarsh_theory_1986}) respectively give
\[
\mu\left(1 - \frac{n}{2^n - 2}\right) \le \frac{1}{2^n - 2},\qquad \mu\left(1 - \frac{1}{2^{n - 1}}\right) \le \frac{1}{(n + 1)2^{n - 1}},
\]
for any integer $n \ge 3$. The classical van der Corput estimate has also been improved by Phillips \cite{phillips_zeta-function_1933} and Graham--Kolesnik \cite[Thm.\ 4.2]{graham_van_1991}. Bounds for other values of $\sigma$ can then be obtained using the convexity estimate: for fixed $\sigma_1 < \sigma_2$, 
\[
\mu(\lambda \sigma_1 + (1 - \lambda)\sigma_2) \le \lambda \mu(\sigma_1) + (1 - \lambda)\mu(\sigma_2),\qquad 0 \le \lambda \le 1. 
\]
Lastly, for $\sigma$ close to 1, the best-known bounds on $\mu(\sigma)$ take the form
\begin{equation}\label{hb_zeta_bound}
\mu(\sigma) \le B(1 - \sigma)^{3/2}.
\end{equation}
Heath-Brown \cite{heathbrown_new_2017} showed that \eqref{hb_zeta_bound} holds with $B = 8\sqrt{15}/63 = 0.4918\ldots$ and $1/2 \le \sigma \le 1$. By optimising the choice of exponent pair, we are able to obtain results that contains all of the above. 
\begin{theorem}\label{mu_est_thm}
We have 
\[
\mu(\sigma) \le \begin{cases}
	 (31 - 36\sigma)/84 , & \frac{1}{2} \leq\sigma < \frac{88225}{153852} = 0.5734\ldots, \\
	 (220633 - 251324\sigma)/620612 , & \frac{88225}{153852} \leq\sigma < \frac{521}{796} = 0.6545\ldots, \\
	 (1333 - 1508\sigma)/3825 , & \frac{521}{796} \leq\sigma < \frac{53141}{76066} = 0.6986\ldots, \\
	 (405 - 454\sigma)/1202 , & \frac{53141}{76066} \leq\sigma < \frac{3620}{5119} = 0.7071\ldots, \\
	 (49318855 - 52938216\sigma)/170145110 , & \frac{3620}{5119} \leq\sigma < \frac{52209}{69128} = 0.7552\ldots, \\
	 (471957 - 502648\sigma)/1682490 , & \frac{52209}{69128} \leq\sigma < \frac{1389}{1736} = 0.8001\ldots, \\
	 (2841 - 3016\sigma)/10316 , & \frac{1389}{1736} \leq\sigma < \frac{134765}{163248} = 0.8255\ldots, \\
	 (859 - 908\sigma)/3214 , & \frac{134765}{163248} \leq\sigma < \frac{18193}{21906} = 0.8305\ldots, \\
	 5(8707 - 9067\sigma)/180277 , & \frac{18193}{21906} \leq\sigma < \frac{249}{280} = 0.8892\ldots, \\
	 (29 - 30\sigma)/130 , & \frac{249}{280} \leq\sigma \leq \frac{9}{10}.\\
\end{cases}
\]
\end{theorem}
In fact the last bound holds for the larger range $249/280 \le \sigma \le 277/300$, however for $\sigma > 9/10$ it is surpassed by the following, which is an improvement of \cite{heathbrown_new_2017}.
\begin{theorem}\label{zeta_bound_thm}
For $1/2 \le \sigma \le 1$, we have 
\[
\mu(\sigma) \le \frac{2}{13}\sqrt{10}(1 - \sigma)^{3/2} = 0.4865\ldots(1 - \sigma)^{3/2}.
\]
\end{theorem}

As remarked in \cite{heathbrown_new_2017} and proved in \cite{bellotti_generalised_2023}, by restricting the range of $\sigma$ to be sufficiently close to 1, one can take $B = 2/3^{3/2} + \delta$ for any $\delta > 0$. We make this explicit when $\sigma$ is very close to 1. 
\begin{theorem}\label{zeta_bound_thm1}
We have 
\[
\mu(\sigma) \le \frac{2}{3^{3/2}}(1 - \sigma)^{3/2} + \frac{103}{300}(1 - \sigma)^{2},\qquad \frac{117955}{118272} \le \sigma \le 1.
\]
\end{theorem}
However, for certain applications we require a uniform bound on the entire range $1/2 \le \sigma \le 1$. One example is the zero-density estimate Corollary \ref{zero_density_bound_thm} in the next section.

\subsection{Zero-density estimates for $\zeta(s)$}
Let $N(\sigma, T)$ denote the number of zeroes of $\zeta(s)$ in the rectangle 
\[
\sigma \le \Re s \le 1,\qquad 0 < \Im s \le T.
\]
A zero-density estimate is a bound on $N(\sigma, T)$ as $T \to \infty$ that holds uniformly for some range of $\sigma$. Results in zero-density estimates are in part motivated by their implications for prime number distributions in short intervals. The well-known density hypothesis that $N(\sigma, T) \ll_{\varepsilon} T^{2(1 - \sigma) + \varepsilon}$ for $1/2 \le \sigma \le 1$ implies an asymptotic formula for the number of primes in $(x, x + O(x^{1/2 + \varepsilon})]$ for any $\varepsilon > 0$. Currently, the density hypothesis is known to hold for $\sigma \ge 25/32$, and many zero-density bounds are known to hold for various ranges of $\sigma$. In Table \ref{zero_density_estimates_table} we record some results of the form 
\[
N(\sigma, T) \ll_{\varepsilon} T^{A(\sigma)(1 - \sigma) + \varepsilon},
\]
for some functions $A(\sigma)$, which, to the best of the authors' knowledge, represent the sharpest known published zero-density estimates for $\zeta(s)$. 
\begin{table}[h]
\def\arraystretch{1.3}
\centering
\caption{Zero-density estimates of the form $N(\sigma, T) \ll_{\varepsilon} T^{A(\sigma)(1 - \sigma) + \varepsilon}$}
\begin{tabular}{|c|c|c|}
\hline
$A(\sigma)$ & Range & Reference\\
\hline
$\dfrac{1 - (8/7 - \delta)(\sigma - 1/2)}{1 - \sigma}$
& \begin{tabular}{@{}c@{}}$\frac{1}{2} \le \sigma \le \frac{1}{2} + o(1)$\\(for any $\delta > 0$)\end{tabular} & Conrey \cite{conrey_at_1989}\\
\hline
$3/(2 - \sigma)$ & $\frac{1}{2} + o(1) < \sigma \le \frac{3}{4} = 0.75$ & Ingham \cite{ingham_estimation_1940}\\
\hline
$3/(7\sigma - 4)$ & $\frac{3}{4} < \sigma < \frac{13}{17} = 0.7647\ldots$ & \multirow{2}{*}{Ivi\'{c} \cite[Ch.\ 11]{ivic_riemann_2003}} \\
\cline{1-2}
$6/(5\sigma - 1)$ & $\frac{13}{17} \le \sigma < \frac{25}{32} = 0.78125$ & \\
\hline
$2$ & $\frac{25}{32} \le \sigma \le \frac{11}{14} = 0.7857\ldots$ & Bourgain \cite{bourgain_large_2000} \\
\hline 
$9/(7\sigma - 1)$ & $\frac{11}{14} < \sigma < \frac{3831}{4791} = 0.7996\ldots$ & Heath-Brown \cite{heathbrown_zero_1979} \\
\hline 
$3/(2\sigma)$ & $\frac{3831}{4791} \le \sigma < \frac{7}{8} = 0.875$ & Ivi\'{c} \cite{ivic_exponent_1980} \\
\hline 
$3/(10\sigma - 7)$ & $\frac{7}{8} \le \sigma < \frac{279}{314} = 0.8885\ldots$ & Heath-Brown \cite{heathbrown_zero_1979}\\
\hline 
\multirow{2}{*}{$24/(30\sigma - 11)$} & $\frac{279}{314} \le \sigma < \frac{155}{174} = 0.8908\ldots$ & \begin{tabular}{@{}c@{}}Chen--Debruyne--Vindas\\\cite{chen_density_2024}\end{tabular} \\
\cline{2-3}
& $\frac{155}{174} \le \sigma \le \frac{49}{54} = 0.9074\ldots$ & Ivi\'{c} \cite{ivic_exponent_1980}\\
\hline 
$2/(7\sigma - 5)$ & $\frac{49}{54} < \sigma \le \frac{15}{16} = 0.9375$ & \multirow{2}{*}{Bourgain \cite{bourgain_remarks_1995}}\\
\cline{1-2}
$4/(30\sigma - 25)$ & $\frac{15}{16} < \sigma \le \frac{23}{24} = 0.9583\ldots$ & \\
\hline 
$3/(24\sigma - 20)$ & $\frac{23}{24} < \sigma < \frac{39}{40} = 0.975$ & \multirow{4}{*}{Pintz \cite{pintz_density_2023}}\\
\cline{1-2}
$2/(15\sigma - 12)$ & $\frac{39}{40} \leq \sigma < \frac{41}{42} = 0.9761\ldots$ & \\
\cline{1-2}
$3/(40 \sigma - 35)$ & $\frac{41}{42} \leq \sigma < \frac{59}{60} = 0.9833\ldots$ & \\
\cline{1-2}
$\dfrac{3}{n(1 - 2(n - 1)(1 - \sigma))}$ & 
\begin{tabular}{@{}c@{}}$1 - \frac{1}{2n(n - 1)} \le \sigma < 1 - \frac{1}{2n(n + 1)}$\\(for integer $n \ge 6$)\end{tabular} & \\
\hline 
\end{tabular}
\label{zero_density_estimates_table}
\end{table}
We also note the currently unpublished work of Kerr \cite{kerr_large_2019} who improved Table \ref{zero_density_estimates_table} in the following ranges
\[
A(\sigma) \le \begin{cases}
36/(138\sigma - 89),& 41/54 \le \sigma < 13/17,\\
3/(2\sigma),& 23/29 \le \sigma < 3831/4791,
\end{cases}
\]
as well as the recent breakthrough work of Guth and Maynard \cite{guth_large_2024}, who established
\[
A(\sigma) \le 15/(3 + 5\sigma),\qquad 1/2 \le \sigma \le 1.
\]

Estimates in Table \ref{zero_density_estimates_table} employ a wide range of different techniques. Historically, estimates close to $\sigma = 1$ relied on bounds on $\zeta(s)$ in the critical strip and a theorem due to Montgomery \cite{montgomery_topics_1971}. The next two corollaries use a similar approach. We note that both results are weaker than those in Bourgain \cite{bourgain_remarks_1995} and Pintz \cite{pintz_density_2023}, so the interest in them lie solely in the optimisation method used to obtain the estimates. 

\begin{corollary}\label{zero_density_bound_thm1}
For any $\varepsilon > 0$, we have $N(\sigma, T) \ll_{\varepsilon} T^{A(\sigma)(1 - \sigma) + \varepsilon}$, where 
\[
A(\sigma) = \begin{cases}
\vspace{1.5mm}
\displaystyle\frac{715(15357 \sigma - 12359)}{(5119 \sigma - 3620)(10238 \sigma - 8739)},&\dfrac{9}{10} \le \sigma < \sigma_1 = 0.9573\ldots\\
\vspace{1.5mm}
\displaystyle\frac{75872(103692 \sigma - 86773)}{5(69128\sigma - 52209) (138256\sigma - 121337)},&\sigma_1 \le \sigma < \sigma_2 = 0.9621\ldots\\
\vspace{1.5mm}
\displaystyle\frac{288(2604 \sigma - 2257)}{(1736 \sigma - 1389)(3472 \sigma - 3125)},&\sigma_2 \le \sigma < \sigma_3 = 0.9644\ldots\\
\vspace{1.5mm}
\displaystyle\frac{22232(244872 \sigma - 216389)}{(163248 \sigma - 134765)(326496 \sigma - 298013)},&\sigma_3 \le \sigma < \sigma_4 = 0.9669\ldots\\
\displaystyle\frac{2860(32859 \sigma - 29146)}{(21906 \sigma - 18193)(43812 \sigma - 40099)},&\sigma_4 \le \sigma \le 1.
\end{cases}
\]
\end{corollary}

\begin{corollary}\label{zero_density_bound_thm}
Uniformly for $9/10 \le \sigma \le 1$, we have 
\[
N(\sigma, T) \ll_{\varepsilon} T^{6.346(1 - \sigma)^{3/2} + \varepsilon}
\]
for any $\varepsilon > 0$.
\end{corollary}

\subsection{The generalised Dirichlet divisor problem}
For integer $k \ge 2$ and $n \ge 1$, if
\[
d_k(n) := \sum_{n_1n_2\cdots n_k = n}1
\]
denotes the $k$-fold divisor function, then the generalised divisor problem concerns bounding the quantity
\[
\Delta_k(x) := \sum_{n \le x}d_k(n) - \underset{s = 1}{\text{Res}}\left(\zeta^k(s)\frac{x^{s}}{s}\right) = \sum_{n \le x}d_k(n) - x P_{k - 1}(\log x),
\]
where $P_{k - 1}$ is a certain polynomial of degree $k - 1$. The conjectured order $\Delta_k(x)$ is $\Delta_k(x) \ll_{\varepsilon} x^{1/2 - 1/(2k) + \varepsilon}$ for any $\varepsilon > 0$, and indeed it is known that $\Delta_k(x) = \Omega(x^{1/2 - 1/(2k)})$ \cite{hardy_dirichlets_1917, szeg_uber_1927, szeg_uber_1927-1}. This problem has been studied extensively, see for example \cite{voronoi_sur_1903, hardy_approximate_1923, titchmarsh_theory_1986, richert_einfuhrung_1960, karacuba_uniform_1972, heathbrown_mean_1981, kolesnik_estimation_1981, ivic_dirichlet_1989, heathbrown_new_2017, bellotti_generalised_2023}. It is known that for $k \ge 4$ and any $\varepsilon > 0$, we have   
\[
\Delta_k(x) \ll_{\varepsilon} x^{\alpha_k + \varepsilon}
\]
for some $\alpha_k$ satisfying
\[
\alpha_k \le \frac{3k - 4}{4k},\qquad (4 \le k \le 8),
\]
\begin{alignat*}{4}
\alpha_9 &\le 35/54,\qquad &&\alpha_{10}\le 27/40,\qquad &&\alpha_{11} \le 0.6957, \qquad &&\alpha_{12} \le 0.7130,\\
\alpha_{13} &\le 0.7306,\qquad &&\alpha_{14}\le 0.7461,\qquad &&\alpha_{15} \le 0.75851,\qquad &&\alpha_{16} \le 0.7691,\\
\alpha_{17} &\le 0.7785,\qquad &&\alpha_{18}\le 0.7868,\qquad &&\alpha_{19} \le 0.7942,\qquad &&\alpha_{20} \le 0.8009,
\end{alignat*}
\[
\alpha_{k} \le 1 - \frac{4}{k + 2},\quad (21 \le k \le 25),\qquad \alpha_k \le 1 - \frac{5}{k + 4},\quad (26 \le k \le 29),
\]
\[
\alpha_k \le 1 - 1.224(k - 8.37)^{-2/3},\qquad (k \ge 30).
\]
The results for $4 \le k \le 8$ are due to \cite{heathbrown_mean_1981}; $k = 9$ is due to \cite[Ch.\ 13]{ivic_riemann_2003}; $10 \le k \le 20$ are due to \cite{ivic_dirichlet_1989} and $21 \le k \le 29$ are due to \cite[Ch.\ 13]{ivic_riemann_2003}. The result for $k \ge 30$ is due to \cite{bellotti_generalised_2023}, who also showed that $\alpha_k \le 1 - 1.889k^{-2/3}$ for sufficiently large $k$. 

Such estimates are obtained via lower bounds on $m(\sigma)$, defined (for each fixed $\sigma$) as the supremum of all numbers $m \ge 4$ for which 
\begin{equation}\label{m_sigma_defn}
\int_1^T|\zeta(\sigma + it)|^m\text{d}t \ll_{\varepsilon} T^{1 + \varepsilon}
\end{equation}
for any $\varepsilon > 0$. The results for $9 \le k \le 29$ in particular depend on an estimate of $m(\sigma)$ developed in \cite[Ch.\ 13]{ivic_riemann_2003} and refined in \cite{ivic_dirichlet_1989}. As remarked in \cite{ivic_dirichlet_1989}, completely optimising the choice of exponent pair in this method requires manipulating unwieldy expressions. By approaching the problem as a constrained optimisation, we obtain the following new estimates for $\alpha_k$ ($9 \le k \le 21$).

\begin{theorem}\label{divisor_problem_thm}
We have 
\begin{align*}
\alpha_{ 9 } &\le 0.64720 , \quad
\alpha_{ 10 } \le 0.67173 , \quad
\alpha_{ 11 } \le 0.69156 , \quad
\alpha_{ 12 } \le 0.70818 , \quad\\
\alpha_{ 13 } &\le 0.72350 , \quad
\alpha_{ 14 } \le 0.73696 , \quad
\alpha_{ 15 } \le 0.74886 , \quad
\alpha_{ 16 } \le 0.75952 , \quad\\
\alpha_{ 17 } &\le 0.76920 , \quad
\alpha_{ 18 } \le 0.77792 , \quad
\alpha_{ 19 } \le 0.78581 , \quad
\alpha_{ 20 } \le 0.79297 , \quad\\
\alpha_{ 21 } &\le 0.79951.
\end{align*}
\end{theorem}

\begin{comment}
\subsection{Gaps between consecutive zeroes of $\zeta(s)$ on the critical line}
Let $\gamma_n$ denote the ordinate of the $n$th zero of $\zeta(s)$ on the critical line (counted with multiplicity). A well-known open problem is the size of the gap $\gamma_{n + 1} - \gamma_n$ as $n\to\infty$. It is known unconditionally that, for any $\varepsilon > 0$,
\[
\gamma_{n + 1} - \gamma_n \ll_{\varepsilon} \gamma_n^{\theta + \varepsilon}
\]
holds with 
\[
\theta = \frac{1}{4},\quad \frac{61}{296},\quad\frac{1}{6},\quad \frac{5}{32},\quad 0.1559458\ldots,\quad \frac{445}{2884}
\]
by Hardy and Littlewood \cite{hardy_contributions_1916}, Moser \cite{moser_on_1976}, Balasubramanian \cite{balasubramanian_improvement_1978}, Karatsuba \cite{karatsuba_distance_1981}, Ivi\'{c} \cite[Thm.\ 10.2]{ivic_riemann_2003}, \cite[Thm.\ 21.4.4]{huxley_area_1996} respectively. The value of $\theta$ is particularly interesting since historically it has been lower than the best-known exponent in the order of $\zeta(1/2 + it)$. Here, we show that
\begin{theorem}\label{gamma_gaps_thm}
If $\gamma_n$ denotes the ordinate of the $n$th zero of $\zeta(s)$ on the critical line, counted with multiplicity, then 
\[
\gamma_{n + 1} - \gamma_n \ll_{\varepsilon} \gamma_n^{\theta + \varepsilon},\qquad \theta = \frac{17}{110} = 0.1545\ldots.
\]
\end{theorem}
\end{comment}

\subsection{The number of primitive Pythagorean triangles}
For our last example we take a brief excursion into a different area, to demonstrate the breath of application of exponent pairs. Let $P(N)$ denote the number of primitive Pythagorean triangles with area no greater than $N$. It is known that 
\[
P(N) = cN^{1/2} - c'N^{1/3} + R(N)
\]
for some suitable constants $c$, $c' > 0$. The remainder term $R(N)$ was bounded to $O(N^{1/3})$ by Lambek and Moser \cite{lambek_distribution_1955}, to $O(N^{1/4}\log N)$ by Wild \cite{wild_number_1955} and finally to 
\[
R(N) \ll N^{1/4}\exp(-\gamma \sqrt{\log N})
\]
for some $\gamma > 0$ by Duttlinger and Schwarz \cite{duttlinger_uber_1980}. This remains the best published unconditional bound on $R(N)$. 

Under the Riemann Hypothesis, sharper bounds are possible. Menzer \cite{menzer_number_1986} has shown that 
\[
R(N) \ll N^{1703927/7513108 + \varepsilon} = N^{0.22679\ldots + \varepsilon}.
\]
We improve this estimate by showing that

\begin{theorem}\label{pythagorean_triple_thm}
Assuming the Riemann Hypothesis, for any $\varepsilon > 0$ we have
\begin{equation}
R(N) \ll_{\varepsilon} N^{71/316 + \varepsilon},
\end{equation}
where $71/316 = 0.22468\ldots$.
\end{theorem}

We close this section by noting that we have not intended to make an exhaustive list of the many applications of exponent pairs and to trace through the corresponding improvements. It is also worth noting that we have only focused on the theory of one-dimensional exponential sums, whereas many applications require consideration of multi-dimensional sums. Examples of such applications include the Piatetski-Shapiro prime number theorem \cite{pyatetskii_on_1953, rivat_prime_2001}, the number of semi-primes in short intervals \cite{wu_almost_2010} and the distribution of square-free numbers \cite{liu_distribution_2016}. Compared to their one-dimensional versions, multi-dimensional sums are substantially more difficult to treat. The articles of Srinivasan \cite{srinivasan_lattice_1963_1, srinivasan_lattice_1963_2, srinivasan_lattice_1965} develop a theory of multi-dimensional exponent pairs. We believe that the methods of this paper can be generalised to higher dimensions, given sufficient effort.

\section{Proof of Lemma \ref{new_exponent_pairs_lem}}
Let $\sigma > 0$ be fixed. We will show that, for each $(k, \ell)$ in the statement of Lemma~\ref{new_exponent_pairs_lem}, there exists $P$ sufficiently large and $c > 0$ sufficiently small, such that uniformly for $f \in \textbf{F}(N, P, \sigma, y, c)$, we have
\[
S = \sum_{n \in I}e(f(n)) \ll \left(\frac{y}{N^\sigma}\right)^k N^\ell, \qquad (y \ge N^\sigma),
\]
provided that $N$ and $y$ are sufficiently large. Throughout, we will only consider the case $I = (N, 2N]$, since the result may be generalised to intervals $I = (a, b] \subseteq (N, 2N]$ using the argument in Sargos \cite[p 310]{sargos_points_1995}. Let 
\[
T := yN^{1 - \sigma},\qquad \alpha := \frac{\log N}{\log T},\qquad F(u) := T^{-1} f(uN)
\]
so that 
\[
S = \sum_{N < n \le 2N}e\left(T F\left(\frac{n}{N}\right)\right).
\]
First, we record some bounds of the form $S \ll_{\varepsilon} T^{\beta(\alpha) + \varepsilon}$ for $\alpha$ in suitable ranges. To show that $(k + \varepsilon, \ell + \varepsilon)$ is an exponent pair, it suffices to show that $\beta(\alpha) \le k + (\ell - k)\alpha$ holds for $0 \le \alpha \le 1/2$. This is because the range $1/2 < \alpha \le 1$ may be handled analogously by first applying Poisson summation (see e.g.\ \cite[p 370]{huxley_area_1996}). There is no need to consider $\alpha > 1$ since $y < N^\sigma$ in this range.  

The exponent pair $(\frac{18}{199} + \varepsilon, \frac{593}{796} + \varepsilon)$ follows directly by taking $(p, q) = (\frac{13}{84} + \varepsilon, \frac{55}{84} + \varepsilon)$ in Sargos \cite[Thm.\ 7.1]{sargos_points_1995}, which implies
\begin{equation}\label{sargoes_exp_pair}
\beta(\alpha) \le \tfrac{18}{199} + \tfrac{521}{796}\alpha,\qquad \left(0 \le \alpha \le \tfrac{1}{2}\right).
\end{equation}
The other exponent pairs are generated by first computing the best known bound on $\beta(\alpha)$ for each $0 \le \alpha \le 1/2$, say $\beta_0(\alpha)$. This will be a piecewise-defined function. Then, we compute the minimal convex region $R \subset \mathbb{R}^2$ containing the points 
\begin{equation}\label{alpha_R_points}
\{(\alpha, \beta) : 0 \le \alpha \le \tfrac{1}{2}, 0 \le \beta \le \beta_0(\alpha) \}.
\end{equation}
Exponent pairs correspond to (non-trivial) tangent lines to this convex region. Intuitively, the set of exponent pairs is isomorphic to the dual of $R$. 

Table \ref{huxley_table_ranges} shows bounds of the form 
\[
\beta(\alpha) \le A + B\alpha,\qquad (X \le \alpha \le Y)
\]
which, to our knowledge, are the sharpest available bounds for each range of $\alpha$ (with the exception of the first range $0 \le \alpha \le \frac{2848}{12173}$, where sharper bounds are possible due to other exponent pairs, however this region does not affect the argument). In the application of the exponential sum estimates in Table~\ref{huxley_table_ranges}, care needs to be taken to ensure that each stated result holds uniformly for all $f \in \textbf{F}(N, P, \sigma, y, c)$. In the next few sections we verify that this is indeed the case for each bound in Table \ref{huxley_table_ranges}.
\begin{table}[h]
\def\arraystretch{1.3}
\centering
\caption{Bounds on $\beta(\alpha)$ of the form $\beta(\alpha) \le \beta_0(\alpha), \;(X \le \alpha \le Y)$}
\begin{tabular}{|c|c|c|c|c|}
\hline
$\beta_0(\alpha)$ & $X$ & $Y$ & Reference\\
\hline
$ \frac{13}{414} + \frac{359}{414} \alpha $ & $ 0 $ & $ \frac{2848}{12173} = 0.2339\ldots $ & Exponent pair $A^2(\frac{13}{84} + \varepsilon, \frac{55}{84} + \varepsilon)$\\
\hline
$\frac{13}{318} + \frac{253}{318} \alpha$ & $\frac{2848}{12173}$ & $\frac{161}{646} = 0.2492\ldots$ & \multirow{6}{*}{Huxley \cite[Table 17.1]{huxley_area_1996}}\\
\cline{1-3}
$\frac{11}{492} + \frac{107}{123} \alpha$ & $\frac{161}{646}$ & $\frac{19}{74} = 0.2567\ldots$ & \\
\cline{1-3}
$ \frac{89}{2706} + \frac{2243}{2706} \alpha $ & $ \frac{19}{74} $ & $ \frac{199}{716} = 0.2779\ldots $ & \\
\cline{1-3}
$ \frac{29}{600} + \frac{58}{75} \alpha $ & $ \frac{199}{716} $ & $ \frac{967}{3428} = 0.2820\ldots $ & \\
\cline{1-3}
$ \frac{49}{1614} + \frac{1351}{1614} \alpha $ & $ \frac{967}{3428} $ & $ \frac{120}{419} = 0.2863\ldots $ & \\
\cline{1-3}
$ \frac{1}{66} + \frac{235}{264} \alpha $ & $ \frac{120}{419} $ & $ \frac{1328}{4447} = 0.2986\ldots $ & \\
\hline
$ \frac{13}{194} + \frac{139}{194} \alpha $ & $ \frac{1328}{4447} $ & $ \frac{104}{343} = 0.3032\ldots $ & Exponent pair $A(\frac{13}{84} + \varepsilon, \frac{55}{84} + \varepsilon)$\\
\hline
$ \frac{13}{146} + \frac{47}{73} \alpha $ & $ \frac{104}{343} $ & $ \frac{87}{275} = 0.3163\ldots $ & \multirow{5}{*}{Huxley \cite[Table 17.1]{huxley_area_1996}}\\
\cline{1-3}
$ \frac{11}{244} + \frac{191}{244} \alpha $ & $ \frac{87}{275} $ & $ \frac{423}{1295} = 0.3266\ldots $ & \\
\cline{1-3}
$ \frac{89}{1282} + \frac{454}{641} \alpha $ & $ \frac{423}{1295} $ & $ \frac{227}{601} = 0.3777\ldots $ & \\
\cline{1-3}
$ \frac{29}{280} + \frac{173}{280} \alpha $ & $ \frac{227}{601} $ & $ \frac{12}{31} = 0.3870\ldots $ & \\
\cline{1-3}
$ \frac{1}{32} + \frac{103}{128} \alpha $ & $ \frac{12}{31} $ & $ \frac{1508}{3825} = 0.3942\ldots $ & \\
\hline
$ \frac{18}{199} + \frac{521}{796} \alpha $ & $ \frac{1508}{3825} $ & $ \frac{62831}{155153} = 0.4049\ldots $ & \eqref{sargoes_exp_pair} and Sargos \cite{sargos_points_1995}\\
\hline
$ \frac{569}{2800} + \frac{1053}{2800} \alpha $ & $ \frac{62831}{155153} $ & $ \frac{143}{349} = 0.4097\ldots $ & \multirow{3}{*}{Huxley \cite[Table 19.2]{huxley_area_1996}}\\
\cline{1-3}
$ \frac{491}{5530} + \frac{1812}{2765} \alpha $ & $ \frac{143}{349} $ & $ \frac{263}{638} = 0.4122\ldots $ & \\
\cline{1-3}
$ \frac{113}{1345} + \frac{897}{1345} \alpha $ & $ \frac{263}{638} $ & $ \frac{1673}{4038} = 0.4143\ldots $ & \\
\hline
$ \frac{2}{9} + \frac{1}{3} \alpha $ & $ \frac{1673}{4038} $ & $ \frac{5}{12} = 0.4166\ldots $ & \multirow{2}{*}{Bourgain \cite[Eqn.\ 3.18]{bourgain_decoupling_2016}}\\
\cline{1-3}
$ \frac{1}{12} + \frac{2}{3} \alpha $ & $ \frac{5}{12} $ & $ \frac{3}{7} = 0.4285\ldots $ & \\
\hline
$ \frac{13}{84} + \frac{1}{2} \alpha $ & $ \frac{3}{7} $ & $ \frac{1}{2}$ &  Bourgain \cite[Thm.\ 4]{bourgain_decoupling_2016}\\
\hline
\end{tabular}
\label{huxley_table_ranges}
\end{table}

\subsubsection{Bounds on $\beta(\alpha)$ from \cite[Eqn.\ (3.13)]{bourgain_decoupling_2016} and \cite[Thm.\ 4]{bourgain_decoupling_2016}}
These results assume only that
\[
\min_{1 \le x \le 2}\{|F''(x)|, |F'''(x)|, |F^{(4)}(x)|\} > c_1
\]
for some constant $c_1 \in (0, 1]$. This holds for all $f \in \textbf{F}(N, P, \sigma, y, c)$ with $P \ge 4$ and $c < 1/2$, since then for $2 \le k \le 4$, 
\begin{equation}\label{Fk_cond}
|F^{(k)}(x)| = T^{-1}N^k|f^{(k)}(xN)| \ge (\sigma)_{k - 2}(1 - c)x^{-k}
\end{equation}
where $(\sigma)_p = \prod_{m = 0}^{p}(\sigma + m)$. 

\subsubsection{Bounds on $\beta(\alpha)$ from \cite[Table 17.1]{huxley_area_1996}}
These bounds follow directly from \cite[Thm.\ 17.1.4]{huxley_area_1996} and \cite[Thm.\ 17.4.2]{huxley_area_1996} with $R = 1, 2$. For $\alpha \le 1/2$, the first of these theorems only requires $F'''(x), F^{(4)}(x) \asymp 1$, which follow immediately from an argument similar to \eqref{Fk_cond}. Next, we will verify that the hypothesis of the second theorem is satisfied for all $f \in \textbf{F}(N, P, \sigma, y, c)$, provided that
\begin{equation}\label{Pc_assumptions}
P \ge 8,\qquad 0 < c < \frac{1}{1000(\sigma + 1)}.
\end{equation}
First, the condition \cite[Eqn.\ (17.4.3)]{huxley_area_1996} follows from \eqref{Fk_cond}, where we note that the lower bound on $P$ ensures $F^{(r)}(x)$ is well-defined for $R + 2 \le r \le R + 4$. 

Next, for $3 \le r \le 5$,
\[
\frac{\text{d}^2}{\text{d}x^2}\log F^{(r)}(x) = \frac{F^{(r)}(x)F^{(r + 2)}(x) - (F^{(r + 1)}(x))^2}{(F^{(r)}(x))^2} \neq 0
\]
since 
\[
\frac{(\sigma)_{r - 2}(\sigma)_{r}}{(\sigma)_{r - 1}^2} = \frac{\sigma + r}{\sigma + r - 1} \ge \frac{\sigma + 5}{\sigma + 4} > \frac{(1 + c)^2}{(1 - c)^2}
\]
by virtue of the upper-bound on $c$ in \eqref{Pc_assumptions}. Thus \cite[Eqn.\ (17.4.4)]{huxley_area_1996} is satisfied. We may verify the remaining two conditions of \cite[Thm.\ 17.4.2]{huxley_area_1996} in a similar fashion. The first condition follows from
\[
\frac{3(\sigma)_{R + 1}}{(\sigma)_{R}(\sigma)_{R + 2}} = \frac{3(\sigma + R + 1)}{\sigma + R + 2} > \frac{(1 + c)^2}{(1 - c)^2},\qquad (R = 1, 2),
\]
which implies $3(F^{(R + 3)})^2 - F^{(R + 2)}F^{(R + 4)} \neq 0$. The second condition also holds, since $F^{(R + 3)}F^{(R + 5)}, F^{(R + 2)}F^{(R + 4)} > 0$, so
\begin{align*}
&\begin{vmatrix}
3(F^{(R+ 3)})^2 + 4F^{(R + 2)}F^{(R + 4)} & 3F^{(R + 2)}F^{(R + 3)} & (F^{(R + 2)})^2\\
F^{(R + 4)} & F^{(R + 3)} & F^{(R + 2)}\\
F^{(R + 5)} & F^{(R + 4)} & F^{(R + 3)}
\end{vmatrix}\\
&\qquad =3(F^{(R + 3)})^4 - 3(F^{(R + 2)}F^{(R + 4)})^2 \\
&\qquad\qquad - 2F^{(R + 2)}F^{(R + 3)}(F^{(R + 2)}F^{(R + 5)} - F^{(R + 3)}F^{(R + 4)})\\
&\qquad < \big(3(\sigma)_{R + 1}^4(1 + c)^4 - 3(\sigma)_{R}^2(\sigma)_{R + 2}^2(1 - c)^4 \\
&\qquad\qquad + 2(\sigma)_{R}^2(\sigma)_{R + 3}(\sigma)_{R + 1}(1 + c)^4 - 2(\sigma)_{R + 1}^2(\sigma)_{R + 2}(\sigma)_{R}(1 - c)^4\big)x^{-4R - 12}\\
&\qquad < 0,
\end{align*}
where the last inequality follows from 
\[
\frac{3(\sigma + R + 1)^2 + 2(\sigma + R + 2)(\sigma + R + 3)}{3(\sigma + R + 2)^2 + 2(\sigma + R + 1)(\sigma + R + 2)} < \frac{(1 - c)^4}{(1 + c)^4},\qquad (R = 1, 2),
\]
valid for $c < (1000(\sigma + 1))^{-1}$.

\subsubsection{Bounds on $\beta(\alpha)$ from \cite[Table 19.2]{huxley_area_1996}} In addition to the assumptions required for \cite[Table 17.1]{huxley_area_1996}, the results of this table must satisfy the assumptions of \cite[Lem.\ 19.2.1]{huxley_area_1996}, which poses no difficulty in view of \eqref{Fk_cond}.

\subsubsection{Bounds on $\beta(\alpha)$ arising from exponent pairs} Finally, if $(k_0, \ell_0)$ is an exponent pair, then there exists $P_0, c_0 > 0$ such that, for all $f \in \textbf{F}(N, P_0, \sigma, y, c_0)$, $S \ll (T/N)^{k_0}N^{\ell_0}$. Therefore, $\beta(\alpha) \le k_0 + (\ell_0 - k_0)\alpha$ for $0 \le \alpha \le 1/2$, provided that we take $P \ge P_0$ and $c \le c_0$. 

\subsubsection{Constructing exponent pairs}
The convex hull containing points of the form \eqref{alpha_R_points}, where $\beta_0(\alpha)$ is given piecewise by the bounds in Table \ref{huxley_table_ranges}, has the following vertices
\[
\left(0, \tfrac{13}{414}\right),\quad\left(\tfrac{1328}{4447}, \tfrac{2499}{8894}\right),\quad\left(\tfrac{104}{343}, \tfrac{195}{686}\right),\quad \left(\tfrac{227}{601}, \tfrac{405}{1202}\right),
\]
\[
\left(\tfrac{1508}{3825},\tfrac{1333}{3825}\right),\quad \left(\tfrac{62831}{155153}, \tfrac{220633}{620612}\right),\quad\left(\tfrac{3}{7}, \tfrac{31}{84}\right),\quad \left(\tfrac{1}{2}, \tfrac{17}{42}\right).
\]
The claimed exponent pairs then follow from lines joining two consecutive vertices. For instance, from the vertices $\left(\tfrac{62831}{155153}, \tfrac{220633}{620612}\right)$ and $\left(\tfrac{3}{7}, \tfrac{31}{84}\right)$ we may verify that 
\begin{equation*}
\beta(\alpha) \le \beta_0(\alpha) \le \tfrac{4742}{38463} + \tfrac{88225}{153852}\alpha,\qquad \left(0 \le \alpha \le \tfrac{1}{2}\right)
\end{equation*}
which implies the exponent pair $(k, \ell) = (\frac{4742}{38463} + \varepsilon, \frac{35731}{51284} + \varepsilon)$. Repeating this process creates seven exponent pairs, of which four (stated in Lemma \ref{new_exponent_pairs_lem}) are new. 

\section{Proof of Theorem \ref{reg_thm_1}}

In this section we prove several results related to the geometry of $H$ (see Definition \ref{convex_hull_H_defn}), which together imply Theorem \ref{reg_thm_1}. 

\begin{lemma}\label{h_theta_convex_lem}
The set $H$ is convex. 
\end{lemma}
\begin{proof}
Due to the symmetry of $H$ about the line $\ell = k + 1/2$, it suffices to show that  the quantity 
\[
Q_n := \frac{\ell_{n + 1} - \ell_n}{k_{n + 1} - k_n},
\]
representing the slope of the line joining two successive vertices of $H$, is negative and decreasing for integer $n \ge 0$, and that $Q_0 \le -1$. We verify this computationally for $0 \le n \le 8$, by explicitly computing $Q_n$. Also, for $n \ge 9$, we have $(k_n, \ell_n) = (p_{n - 4}, q_{n - 4})$ (where $(p_m, q_m)$ are defined in \eqref{hb_exponent_pair}), so
\[
Q_n = \frac{q_{n - 3} - q_{n - 4}}{p_{n - 3} - p_{n - 4}} = -\frac{(n - 5)(n - 4)}{n - 3}
\]
which is negative and decreasing for all $n \ge 9$, as required. 
\end{proof}

Next, we seek to show that $H$ is closed under the $A$ and $C$ transformations. While such transformations are non-linear, it turns out that both $A$ and $C$ satisfy a type of quasilinear property as shown in the next lemma. 
\begin{lemma}\label{quasiconvexity_lem}
Let $P: [0, \frac{1}{2}] \times [\frac{1}{2}, 1] \mapsto [0, \frac{1}{2}] \times [\frac{1}{2}, 1]$ be a projective transformation of the form 
\[
P(k, \ell) := \left(\frac{\phi_1(k, \ell)}{\phi_2(k, \ell)}, \frac{\phi_3(k, \ell)}{\phi_2(k, \ell)}\right),\qquad \phi_i = a_i k + b_i \ell + c_i
\]
where $a_i, b_i, c_i$ are constants. Then, for any $p_1 = (k_1, \ell_1)$ and $p_2 = (k_2, \ell_2)$, we have 
\[
P(\lambda p_1 + (1 - \lambda)p_2) = \mu P(p_1) + (1 - \mu) P(p_2)
\]
for some monotonically increasing $\mu = \mu(\lambda)$ satisfying $\mu(0) = 0$, $\mu(1) = 1$. 
\end{lemma}
\begin{proof}
If $(k, \ell) = P(\lambda p_1 + (1 - \lambda) p_2)$ where $p_1 = (k_1, \ell_1)$, $p_2 = (k_2, \ell_2)$ then, by linearity of $\phi_i$, 
\begin{equation}\label{k_equality}
k = \frac{\phi_1(\lambda k_1 + (1 - \lambda)k_2, \lambda\ell_1 + (1 - \lambda)\ell_2)}{\phi_2(\lambda k_1 + (1 - \lambda)k_2, \lambda\ell_1 + (1 - \lambda)\ell_2)} = \frac{\lambda \phi_1(k_1, \ell_1) + (1 - \lambda)\phi_1(k_2, \ell_2)}{\lambda \phi_2(k_1, \ell_1) + (1 - \lambda)\phi_2(k_2, \ell_2)}.
\end{equation}
Let us define
\[
\mu = \frac{\lambda}{\lambda + (1 - \lambda)\frac{\phi_2(k_2, \ell_2)}{\phi_2(k_1, \ell_1)}}
\]
so that $\mu$ varies monotonically from 0 to 1 with $\lambda$, and 
\[
\frac{\phi_2(k_2, \ell_2)}{\phi_2(k_1, \ell_1)} = \frac{\lambda}{1 - \lambda}\frac{1 - \mu}{\mu}.
\]
This gives
\begin{align*}
&(\lambda \phi_2(k_1, \ell_1) + (1 - \lambda)\phi_2(k_2, \ell_2))\left(\mu \frac{\phi_1(k_1, \ell_1)}{\phi_2(k_1, \ell_1)} + (1 - \mu)\frac{\phi_1(k_2, \ell_2)}{\phi_2(k_2, \ell_2)}\right) \\
& = \lambda \mu \phi_1(k_1, \ell_1) + (1 - \lambda)(1 - \mu)\phi_1(k_2, \ell_2) \\
&\qquad\qquad+ \mu(1 - \lambda)\phi_1(k_1, \ell_1)\frac{\phi_2(k_2, \ell_2)}{\phi_2(k_1, \ell_1)} + \lambda(1 - \mu)\phi_1(k_2, \ell_2)\frac{\phi_2(k_1, \ell_1)}{\phi_2(k_2, \ell_2)}\\
&= \lambda \mu \phi_1(k_1, \ell_1) + (1 - \lambda)(1 - \mu)\phi_1(k_2, \ell_2) + \mu(1 - \lambda)\phi_1(k_2, \ell_2) + \lambda(1 - \mu)\phi_1(k_1, \ell_1)\\
&= \lambda \phi_1(k_1, \ell_1) + (1 - \lambda)\phi_1(k_2, \ell_2).
\end{align*}
Therefore, combining with \eqref{k_equality} gives
\[
k = \mu \frac{\phi_1}{\phi_2}(k_1, \ell_1) + (1 - \mu)\frac{\phi_1}{\phi_2}(k_2, \ell_2).
\]
A similar procedure gives 
\[
\ell = \mu \frac{\phi_3}{\phi_2}(k_1, \ell_1) + (1 - \mu)\frac{\phi_3}{\phi_2}(k_2, \ell_2),
\]
and the desired result follows. 
\end{proof}

Observe that both the $A$ and $C$ transformations are of the form specified in Lemma \ref{quasiconvexity_lem}. 
A corollary is that the image of the line segment joining points $p_1$ and $p_2$ under $A$, is the line segment joining points $A(p_1)$ and $A(p_2)$ (and analogously for the $C$ operation). The observation that $A$ maps line segments to line segments was also noted in \cite{petermann_divisor_1988}, without proof. 

\begin{lemma}\label{H_closed_A_lem}
If $(k, \ell) \in H$, then $A(k, \ell) \in H$ and $C(k, \ell) \in H$. 
\end{lemma}
\begin{proof}
Let $P$ denote a transformation satisfying the conditions of Lemma \ref{quasiconvexity_lem}. By Lemma \ref{quasiconvexity_lem}, the image under $P$ of a convex polygon with vertices $p_1, p_2, \ldots$ is a convex polygon with vertices $P(p_1), P(p_2), \ldots$. To show that $H$ is closed under $P$, it suffices to show that the image of any vertex of $H$ lies inside $H$. That is, for all integers $m$, we seek to show that 
\[
(k_m', \ell_m') := P(k_m, \ell_m) \in H
\]
where $(k_m, \ell_m)$ defined in \eqref{knln_defn} are the vertices of $H$. Note that since $k_m' \in [0, 1/2]$, it is necessarily the case that $k_{N + 1} < k_m' \le k_N$ for some integer $N = N(m)$. Therefore, to show that $P(k_m, \ell_m) \in H$ it suffices to prove that if $k_m' \in (k_{N + 1}, k_N]$ then
\begin{equation}\label{H_condition}
k_m'(\ell_{N + 1} - \ell_N) + \ell_m'(k_N - k_{N + 1}) \ge k_N\ell_{N + 1} - \ell_N k_{N + 1},
\end{equation}
and that
\begin{equation}\label{requirement3}
\ell_m' + k_m' < 1,\qquad m \in \mathbb{Z}.
\end{equation}
These inequalities are obtained by inspecting the boundary of the region $H \cap \{(k, \ell): k_{N + 1} < k \le k_N\}$. For instance, \eqref{H_condition} arises because the line joining $(k_N, \ell_N)$ and $(k_{N + 1}, \ell_{N + 1})$ has equation 
\[
k(\ell_{N + 1} - \ell_N) + \ell(k_N - k_{N + 1}) = k_N\ell_{N + 1} - \ell_N k_{N + 1}.
\]
Let us now specialise our argument to the $A$ transformation, so that 
\[
(k_m', \ell_m') = \left(\frac{k_m}{2k_m + 2}, \frac{\ell_m}{2k_m + 2} + \frac{1}{2}\right).
\]
We computationally verify that $A(k_m, \ell_m) \in H_{1000} \subset H$ for $|m| < 100$, where $H_N$ is defined in \eqref{HN_vertices} and the verification source code is given in \S \ref{verify_program}. For $|m| \ge 100$, observe that since $(k_m, \ell_m) \in H$, we have $k_m + \ell_m < 1$ and so 
\[
\frac{k_m}{2k_m + 2} + \frac{\ell_m}{2k_m + 2} + \frac{1}{2} \le \frac{1}{2k_m + 2} + \frac{1}{2} < 1,
\]
so \eqref{requirement3} holds. Finally, applying Lemma \ref{kmlm_induction_lem} below, we also see that condition \eqref{H_condition} holds. Therefore, $H$ is closed under $A$. 

The argument for the $C$ transformation is similar. Here, we have
\[
(k_m', \ell_m') =  \left(\frac{k_m}{12(1 + 4k_m)}, \frac{\ell_m}{12(1 + 4k_m)} + \frac{11}{12} \right).
\]
If $|m| < 100$ we computationally verify as before that $C(k_m, \ell_m)\in H_{1000} \subset H$. For $|m| \ge 100$, since $k_m + \ell_m < 1$,
\[
\frac{k_m}{12(1 + 4k_m)} + \frac{\ell_m}{12(1 + 4k_m)} < \frac{1}{12}
\]
so that $k_m' + \ell_m' < 1$. Therefore, \eqref{requirement3} is satisfied. To complete the proof we observe that \eqref{H_condition} holds via Lemma \ref{C_transform_lem}. Therefore, $H$ is also closed under $C$. 
\end{proof}

\begin{lemma}\label{kmlm_induction_lem}
Let $(k_n, \ell_n)$ be as defined in \eqref{knln_defn}. If $|m| \ge 100$ and $N$ are integers such that $k_{N + 1} < \frac{k_m}{2k_m + 2} \le k_N$, then
\begin{equation}\label{result_kmlm}
\frac{k_m}{2k_m + 2}(\ell_{N + 1} - \ell_N) + \left(\frac{\ell_{m}}{2k_m + 2} + \frac{1}{2}\right)(k_N - k_{N + 1}) \ge k_N\ell_{N + 1} - \ell_N k_{N + 1}.
\end{equation}
\end{lemma}
\begin{proof}
Note that we necessarily have $N \ge m$, since 
\[
\frac{k_m}{2k_m + 2} < k_m
\]
as $k_m \in [0, 1/2]$, and thus $k_m > k_{N + 1}$, which implies $N + 1 > m$ as $k_n$ is decreasing.  
\subsubsection*{Case 1: $m \ge 100$} 
We have $N \ge m \ge 100$ and also $(k_n, \ell_n) = (p_{n - 4}, q_{n - 4})$ for $n \ge 100$. Thus
\begin{equation}\label{m_lower_bound}
m^{-3} < \frac{1}{(m - 5)^{2}(m - 2) + 2} \le \frac{2}{(N - 5)^{2}(N - 2)} < 2(N - 5)^{-3},
\end{equation}
where the second inequality follows from the assumption $k_m/(2k_m + 2) \le k_{N}$. However,
\begin{align*}
\frac{\ell_m}{2k_m + 2} + \frac{1}{2} &= 1 - \frac{3m^{2} - 27 m + 62}{2(m - 4)(m^{3} - 12 m^{2} + 45 m - 48)} > 1 - \frac{3}{2}(m - 4)^{-2},
\end{align*}
where the RHS is increasing for $m \ge 100$, so using $m > 2^{-1/3}(N - 5)$ from \eqref{m_lower_bound} gives
\begin{equation}\label{lN1_bound}
\frac{\ell_m}{2k_m + 2} + \frac{1}{2} > 1 - \frac{3}{2^{1/3}(N - 2^{7/3} - 5)^2} > 1 - \frac{3N - 11}{(N - 4)(N - 3)(N - 1)} = \ell_{N + 1},
\end{equation}
where the last inequality follows from a direct calculation. Meanwhile, using $k_m /(2k_m + 2) > k_{N + 1}$, we obtain
\[
k_N\ell_{N + 1} - \ell_N k_{N + 1} - \frac{k_m}{2k_m + 2}(\ell_{N + 1} - \ell_N) < \ell_{N + 1}(k_N - k_{N + 1}).
\]
The desired result follows from substituting \eqref{lN1_bound}. 

\subsubsection*{Case 2: $m \le -100$}
For this range of $m$ we have $k_m \in (0.49, \frac{1}{2}]$ so 
\begin{equation}\label{km_bound}
\frac{k_m}{2k_m + 2} \in (0.164, \tfrac{1}{6}) \subset (k_1, k_0)
\end{equation}
and hence $N = 0$. Furthermore, since $k_m \le \frac{1}{2} \le \ell_m$,
\begin{equation}\label{lm_bound}
\frac{\ell_m}{2k_m + 2} + \frac{1}{2} \ge \frac{2}{3}.
\end{equation}
The desired bound follows from substituting into \eqref{result_kmlm} the values $N = 0$, $(k_0, \ell_0) = \left(\frac{13}{84}, \frac{55}{84}\right)$, $(k_1, \ell_1) = \left(\frac{4742}{38463}, \frac{35731}{51284}\right)$, \eqref{km_bound} and \eqref{lm_bound}. 
\end{proof}

\begin{lemma}\label{C_transform_lem}
Let $(k_n, \ell_n)$ be as defined in \eqref{knln_defn}. If $|m| \ge 100$ and $N$ are integers such that $k_{N + 1} < \frac{k_m}{12(1 + 4k_m)} \le k_N$, then
\begin{equation}\label{result_kmlm_C}
\frac{k_m}{12(1 + 4k_m)}(\ell_{N + 1} - \ell_N) + \left(\frac{\ell_{m}}{12(1 + 4k_m)} + \frac{11}{12}\right)(k_N - k_{N + 1}) \ge k_N\ell_{N + 1} - \ell_N k_{N + 1}.
\end{equation}
\end{lemma}
\begin{proof}
Proceeding similarly to the proof of Lemma \ref{kmlm_induction_lem}, consider first the case when $m \ge 100$. 
The bound $\frac{k_m}{12(1 + 4k_m)} \le k_N$ implies, together with $(k_n, \ell_n) = (p_{n - 4}, q_{n - 4})$ for $n \ge 100$, that 
\[
\frac{1}{9}m^{-3} < \frac{1}{6((m - 5)^{2}(m - 2) + 8)} \le \frac{2}{(N - 5)^{2}(N - 2)} < 3N^{-3}
\]
where the first and last inequalities follow from $m \ge 100$ and $N \ge 100$ respectively. This implies $m > N/3$, so that, as before
\begin{align*}
\frac{\ell_{m}}{12(1 + 4k_m)} + \frac{11}{12} &= 1 - \frac{3 m^{2} - 21 m + 38}{12(m - 4)(m^{3} - 12 m^{2} + 45 m - 42)} > 1 - \frac{1}{3m^2}\\
&> 1 - \frac{3}{N^2} \ge 1 - \frac{3N - 11}{(N - 4)(N - 3)(N - 1)} = \ell_{N + 1}.
\end{align*}
The rest of the argument proceeds as per Lemma \ref{kmlm_induction_lem}.

Next suppose $m \le -100$. Then, by \eqref{knln_defn} we have $\frac{k_m}{12(1 + 4k_m)} > \frac{13}{1000}$ and $\frac{\ell_m}{12(1 + 4k_m)} + \frac{11}{12} \ge \frac{67}{72}$. Furthermore, since $k_m \in (0.49, \frac{1}{2}]$, $k_{10} < \frac{k_m}{12(1 + 4k_m)} < k_9$ so that $N(m) = 9$ for all $m \le -100$. The result follows from these bounds and $(k_9, \ell_9) = \left(\frac{1}{56}, \frac{127}{140}\right)$, $(k_{10}, \ell_{10}) = \left(\frac{1}{100}, \frac{14}{15}\right)$.
\end{proof}

\begin{lemma}\label{H_contains_lem}
The set $H$ contains all known exponent pairs of the form \eqref{BI_exp_pair}, \eqref{huxley_pair_1}, \eqref{huxley_pair_2}, \eqref{kth_deriv_test_exp_pair}, \eqref{hb_exponent_pair1} and \eqref{hb_exponent_pair}.
\end{lemma}
\begin{proof}
Recall that for a positive integer $N$, $H_N$ denotes the convex hull of the points 
\[
\{(k_n, \ell_n)\}_{|n| \le N} \cup\{(0, 1), (1/2, 1/2)\}.
\]
Since $H$ is convex by Lemma \ref{h_theta_convex_lem}, we have $H_N \subseteq H$, so to show that $(k, \ell) \in H$ it suffices to show that $(k, \ell) \in H_N$. In Program 1 (\S \ref{verify_program} below), we take $N = 1000$ and verify that $H_N$ contains
\begin{enumerate}
    \item all known exponent pairs of the form \eqref{BI_exp_pair}, \eqref{huxley_pair_1} and \eqref{kth_deriv_test_exp_pair},
    \item exponent pairs of the form \eqref{huxley_pair_2} and \eqref{hb_exponent_pair1} for $m \le 100$,
    \item exponent pairs of the form \eqref{hb_exponent_pair} for $m = 3, 4$.\label{exponent_family_4}
\end{enumerate}
Note that, with the exception of exponent pairs of the form \eqref{hb_exponent_pair1}, we make use of rational numbers in performing this verification so there is no potential for round-off errors. 

For $m \ge 5$, exponent pairs of the form \eqref{hb_exponent_pair} are given by $(p_{m}, q_{m}) = (k_{m + 4}, \ell_{m + 4})$ which lie in $H$ by construction. Thus, it only remains to verify that $H$ also contains exponent pairs of the form \eqref{huxley_pair_2} and \eqref{hb_exponent_pair1} for $m > 100$. We first show that $H$ contains the region 
\[
R = \{(k, \ell): k^{2/3} + \ell \ge 1, k + \ell \le 1, k \le k_{100}\}.
\]
Of the three constraints defining the boundary of $R$, only the first requires further elaboration. Note that if $(k, \ell)$ lies on the boundary of $H$ with $k + \ell < 1$ and $k \le k_{100}$, then 
\[
(k, \ell) = \lambda (k_n, \ell_n) + (1 - \lambda)(k_{n + 1}, \ell_{n + 1}),
\]
for some $\lambda \in [0, 1]$ and integer $n \ge 100$. However, as $(k_n, \ell_n) = (p_{n - 4}, q_{n - 4})$ in this region, 
\[
k \le k_n = \frac{2}{(n - 5)^{2}(n - 2)} < \frac{2^{3/2}}{n^3}
\]
and so
\[
\frac{1 - k^{2/3}}{\ell} > \frac{1 - 2n^{-2}}{\ell_{n + 1}} = 1 + \frac{n^{3} + 5 n^{2} - 38 n + 24}{n^{2}(n^{3} - 8 n^{2} + 16 n - 1)} > 1
\]
for $n \ge 100$. Thus, for all such points we have $k^{2/3} + \ell < 1$, which shows that $R \subseteq H$. 

Next, with $(a_m, b_m)$ as defined in \eqref{hb_exponent_pair1}, observe that $(a_m, b_m) \in R$ for $m > 100$, since
\[
a_m + b_m < 1, \qquad a_m < a_{100} < k_{100},
\]
\[
a_m + b_m^{2/3} = 1 + \frac{1}{(25m^2 (m - 2) \log m)^{2/3}} - \frac{1}{25m^2 \log m} > 1,\qquad (m > 100).
\]
Therefore, $(a_m, b_m) \in H$ for $m \ge 100$. A similar verification confirms that $H$ contains \eqref{huxley_pair_2} for $m > 100$. 
\end{proof}

\section{Proof of theorems}

In this section we prove results related to applications of exponent pairs outlined in \S \ref{sec:applications}. 

\subsection{Proof of Theorem \ref{M_bound_12_thm}}\label{sec5.1}
As usual we proceed by bounding how frequently $\zeta(1/2 + it)$ can be large. Let $1/2 \le \sigma < 1$, $T > 0$, $V > 0$ and suppose $t_1, \ldots, t_R$ are any points satisfying
\[
|\zeta(\sigma + it_r)| \ge V, \qquad |t_r| \le T,\qquad (1 \le r \le R),
\]
\[
|t_r - t_s| \ge 1,\qquad (1 \le r \ne s \le R).
\]
It is well-known that certain bounds on $R$ lead to bounds on moments of $\zeta(s)$. As per \cite[\S 8.1]{ivic_riemann_2003}, the following statements are equivalent
\begin{equation}\label{equiv_formula_1}
\int_1^T|\zeta(\sigma + it)|^{b}\text{d}t \ll_{\varepsilon} T^{a + \varepsilon},
\end{equation}
\begin{equation}\label{equiv_formula_2}
\sum_{r \le R}|\zeta(\sigma + it_r)|^b \ll_{\varepsilon} T^{a + \varepsilon},
\end{equation}
\begin{equation}\label{equiv_formula_3}
R \ll_{\varepsilon} T^{a + \varepsilon}V^{-b + \varepsilon},
\end{equation}
where $a$ and $b$ may depend on $\sigma$. Note that in \eqref{equiv_formula_3} we have $T^{a + \varepsilon}V^{-b + \varepsilon}$ in place of Ivi\'{c}'s $T^{a + \varepsilon}V^{-b}$, which are equivalent since $V \le T$. In fact, if $\zeta(\sigma + it) \ll_{\varepsilon} t^{c(\sigma) + \varepsilon}$ for some $c(\sigma) > 0$, then we may assume throughout that $V \le T^{c(\sigma) + \varepsilon}$, for otherwise $R \ll 1$ and 
\[
\sum_{r \le R}|\zeta(\sigma + it_r)|^A \ll_{\varepsilon} T^{A\, c(\sigma) + \varepsilon}
\]
i.e. $M(A) \le A c(\sigma)$ which is stronger than all of the results of this section. 

\begin{comment}
We reproduce one version below, based on \cite[\S 8.1]{ivic_riemann_2003}, for completeness. 
\begin{lemma}\label{R_bound_zeta_bound}
Let $\varepsilon > 0$ be arbitrary, and let $t_r$, $R$, $T$, $V$ be as defined above. If $R \ll_{\varepsilon} T^{a + \varepsilon}V^{-b + \varepsilon}$ uniformly for $0 < V \le T^{13/84 + \varepsilon}$, then 
\[
\int_1^{T}|\zeta(1/2 + it)|^b\text{d}t \ll_{\varepsilon} T^{a + \varepsilon}.
\]
\end{lemma}
\begin{proof}
For any $Z > 0$, let $t_{V,1}, t_{V,2}, \ldots, t_{Z, R_V}$ be the subset of $\{t_r\}$ satisfying 
\[
\frac{V}{2} < |\zeta(1/2 + it_{V, j})| \le V,\qquad 1 \le j \le R_V.
\]
Then
\[
\sum_{j = 1}^{R_V}|\zeta(1/2 + it_{V, j})|^a \ll V^b ,
\]
and summing over the $O(\log T)$ subsums given by $V = T^{13/84 + \varepsilon}$, $T^{13/84 + \varepsilon} / 2$, $T^{13/84 + \varepsilon} / 4$, $\ldots$,
\[
\sum_{r = 1}^{R}|\zeta(1/2 + it_r)|^a \ll \sum_{V}\sum_{j = 1}^{R_V}|\zeta(1/2 + it_{V, j})|^a \ll_{\varepsilon} R\max_{V}V^b \ll_{\varepsilon} T^{a + \varepsilon},
\]
where we have ensured that all but at most $O(1)$ of the $t_r$'s have been accounted for by $\zeta(1/2 + it)  \ll_{\varepsilon} t^{13/84 + \varepsilon}$ for any $\varepsilon > 0$.
\end{proof}
\end{comment}

The results of this section depend on upper bounds on $R$, such as the following, due to \cite[Thm.\ 8.2]{ivic_riemann_2003}.
\begin{lemma}\label{ivic_exponent_pair_R_bound}
For all exponent pairs $(k, \ell)$ with $k > 0$, and any $\varepsilon > 0$, 
\[
R \ll_\varepsilon T^{1 + \varepsilon}V^{-6} + T^{1 + \ell/k + \varepsilon}V^{-2(1 + 2k + 2\ell)/k}.
\]
\end{lemma}
Taking the sequence of exponent pairs in \eqref{hb_exponent_pair} and applying the $B$ process, we obtain
\[
(k, \ell) = \left(\frac{1}{2} - \frac{3m - 2}{m(m - 1)(m + 2)} + \varepsilon, \frac{1}{2} + \frac{2}{(m - 1)^2(m + 2)}\right),
\]
so that, by Lemma \ref{ivic_exponent_pair_R_bound}, for any integer $m \ge 3$ we have
\[
R \ll_\varepsilon T^{1 + \varepsilon}V^{-6} + T^{\phi_m + \varepsilon}V^{-(12 + \delta_m) + \varepsilon}
\]
where
\[
\phi_m =  2 + \frac{2}{m - 2} - \frac{2}{m - 1} + \frac{4}{m^2 + 3 m - 2},
\]
\[
\delta_m = \frac{12}{m - 2} - \frac{8}{m - 1} - \frac{4 (m - 5)}{m^2 + 3 m - 2}.
\]
Via a routine calculation, we find that for $m \ge 6$ we have 
\[
\delta_m^{-3/2}\left(\phi_m - 2 - \frac{\delta_m}{8}\right) = \frac{m(m^4 - 9 m^2 + 12m - 4)^{1/2}}{32(3m^2 - 4m + 2)^{3/2}} \le \frac{3}{344}\sqrt{\frac{65}{86}}.
\]
Meanwhile, via convexity we also have, for any $\lambda \in [0, 1]$, 
\[
R \ll_\varepsilon T^{1 + \varepsilon}V^{-6} + T^{\lambda\phi_m + (1 - \lambda)\phi_{m + 1}+ \varepsilon}V^{-(12 + \lambda\delta_m + (1 - \lambda)\delta_{m + 1}) + \varepsilon}
\]
so that, via the equivalence of \eqref{equiv_formula_1} and \eqref{equiv_formula_3},
\begin{align*}
&M(12 + \lambda\delta_m + (1 - \lambda)\delta_{m + 1}) \le \lambda\phi_m + (1 - \lambda)\phi_{m + 1} \\
&\qquad\qquad \le 2 + \frac{\lambda \delta_m + (1 - \lambda)\delta_{m + 1}}{8} + \frac{3}{344}\sqrt{\frac{65}{86}}(\lambda \delta_m^{3/2} + (1 - \lambda)\delta_{m + 1}^{3/2}) \\
&\qquad\qquad \le 2 + \frac{\lambda \delta_m + (1 - \lambda)\delta_{m + 1}}{8} + C(\lambda \delta_m + (1 - \lambda)\delta_{m + 1})^{3/2}
\end{align*}
where 
\[
C = \frac{3}{7568}\sqrt{510}.
\]
Since $\delta_m > \delta_{m + 1}$, the last inequality follows from
\begin{align*}
\frac{(\lambda\delta_m^{3/2} + (1 - \lambda)\delta_{m + 1}^{3/2})^{2/3}}{\lambda\delta_m + (1 - \lambda)\delta_{m + 1}} < \frac{\delta_m^{1/3}(\lambda\delta_m + (1 - \lambda)\delta_{m + 1})^{2/3}}{\lambda\delta_m + (1 - \lambda)\delta_{m + 1}} \le \frac{\delta_m^{1/3}}{\delta_{m + 1}^{1/3}} \le \left(\frac{2193}{1573}\right)^{1/3},
\end{align*}
since $m \ge 6$.
Therefore, we have 
\[
M(12 + \delta) \le 2 + \frac{\delta}{8} + C\delta^{3/2},\qquad 0 < \delta \le \frac{86}{65},
\]
which completes the proof. 

\subsection{Proof of Theorem \ref{M_bound_larger_A}}
If $(k, \ell)$ is an exponent pair then by Lemma \ref{ivic_exponent_pair_R_bound} we have
\begin{equation}\label{thm_M_bound_R_bound}
R \ll_{\varepsilon} T^{1 + \varepsilon}V^{-6} + T^{1 + \ell/k + \varepsilon}V^{-2(1 + 2k + 2\ell)/k} \ll_{\varepsilon} T^{1 + \ell/k + \varepsilon}V^{-2(1 + 2k + 2\ell)/k}
\end{equation}
if $V \le T^{\ell/(4\ell - 2k + 2)}$. This is always the case for $V \le T^{13/84 + \varepsilon}$, since 
\[
\frac{\ell}{4\ell - 2k + 2} \ge \frac{\ell}{4\ell + 1} \ge \frac{1}{6}.
\]
Writing
\[
A = 4 + \frac{2 + 4\ell}{k},
\]
it follows from \eqref{thm_M_bound_R_bound} and \eqref{equiv_formula_3} that 
\begin{equation}\label{MA_bound_eqn}
M(A) \le \frac{A}{4} - \frac{1}{2k}.
\end{equation}
Thus the optimisation problem we consider is (for each fixed $A \ge 12$)
\[
\min_{\substack{(k, \ell) \in H\\4 + (2 + 4\ell)/k = A}}\left(\frac{A}{4} - \frac{1}{2k}\right).
\]
It suffices to solve 
\[
\min_{\substack{(k, \ell) \in H\\4 + (2 + 4\ell)/k = A}} k = \min_{(k, \ell) \in H}\frac{2 + 4\ell}{A - 4} = \frac{2}{A - 4} + 4\min_{(k, \ell) \in H}\ell.
\]
Therefore, the solution lies on the boundary of $H$. If $(\kappa_1, \lambda_1)$ and $(\kappa_2, \lambda_2)$ are exponent pairs with $\kappa_2 > \kappa_1$, then by convexity so is 
\[
\left(k, \frac{\lambda_1 - \lambda_2}{\kappa_1 - \kappa_2}(k - \kappa_2) + \lambda_2\right), \qquad \kappa_1 \le k \le \kappa_2.
\]
Substituting this exponent pair into $4 + (2 + 4\ell)/k = A$ gives
\[
k = \frac{2(2\kappa_1\lambda_2 - 2 \kappa_2\lambda_1 + \kappa_1 - \kappa_2)}{(A - 4)(\kappa_1 - \kappa_2) - 4\lambda_1 + 4\lambda_2},
\]
and hence by \eqref{MA_bound_eqn},
\[
M(A) \le \frac{A (\kappa_1\lambda_2 - \kappa_2\lambda_1) + 2 (\kappa_1 - \kappa_2 + \lambda_1 - \lambda_2)}{2(2\kappa_1\lambda_2 - 2\kappa_2\lambda_1 + \kappa_1 - \kappa_2)},\quad 4 + \frac{2 + 4\lambda_2}{\kappa_2} \le A \le 4 + \frac{2 + 4\lambda_1}{\kappa_1}.
\]
As usual let $(k_n, \ell_n)$ denote the vertices of $H$, defined in \eqref{knln_defn}. We take $(\kappa_1, \lambda_1) = (k_{n + 1} + \varepsilon, \ell_{n + 1} + \varepsilon)$ and $(\kappa_2, \lambda_2) = (k_n + \varepsilon, \ell_n + \varepsilon)$ for $-10 \le n \le 1$ which gives the first twelve cases of Theorem \ref{M_bound_larger_A}. For example, in the case $n = 0$ we choose  
\[
(\kappa_1, \lambda_1) = \left(\frac{4742}{38463} + \varepsilon, \frac{35731}{51284} + \varepsilon\right),\qquad (\kappa_2, \lambda_2) = \left(\frac{13}{84} + \varepsilon, \frac{55}{84} + \varepsilon\right),
\]
which gives
\[
M(A) \le \frac{31A - 24}{196},\qquad\frac{440}{13} \le A \le \frac{203087}{4742}.
\]
Here we have used the fact that if $M(A) \le \theta + \varepsilon$ for any $\varepsilon > 0$, then $M(A) \le \theta$. It remains to show that
\begin{equation}\label{large_MA_bound}
M(A) \le 1 + \frac{13}{84}(A - 6)
\end{equation}
for $A > \frac{3516129}{65729}$. To prove \eqref{large_MA_bound} we follow the argument of Ivi\'{c} \cite[Thm.\ 8.3]{ivic_riemann_2003}, with the caveat that the original argument can only produce $M(A) \le 1 + c(A- 6)$ for $c \ge 4/25$. Fortunately, only a small modification is required, and we use this opportunity to generalise Ivi\'{c}'s argument. 

\begin{lemma}
Let $(\sqrt{13} - 3)/4 = 0.15138\ldots \le \theta < 1/4$. If $(\theta + \varepsilon, \theta + 1/2 + \varepsilon)$ is an exponent pair for any $\varepsilon > 0$, then 
\[
M(A) \le 1 + \theta(A - 6)
\]
for all $A \ge 8 + 4/\theta$. 
\end{lemma}
\begin{proof}
First we note that $0.15138\ldots$ is not far from $13/84 = 0.15476\ldots$ and appears to be the current limit of the method. Let $\{t_r\}$, $R$, $V$ and $T$ be as defined in \S \ref{sec5.1}, and suppose $\{\tau_1, \ldots, \tau_{S}\}$ is the subset of $\{t_r\}$ satisfying
\[
Z < |\zeta(1/2 + i\tau_s)| \le 2Z,\qquad 1 \le s \le S
\]
for some $V \le Z \le T$. If $(\theta +\varepsilon, 1/2 + \theta + \varepsilon)$ is an exponent pair, then 
\[
\zeta(1/2 + it) \ll_{\varepsilon} t^{\theta + \varepsilon}.
\]
This is shown in \eqref{upper_bound_zeta_temp} in the proof of Theorem \ref{mu_est_thm} below. We may thus assume that $Z \le T^{\theta + \varepsilon}$. From Lemma \ref{ivic_exponent_pair_R_bound}, if $(k, \ell)$ is an exponent pair, then
\[
S \ll_{\varepsilon} T^{\varepsilon} (TZ^{-6} + T^{(k + \ell)/k}Z^{-2(1 + 2k + 2\ell)/k}).
\]
Numerically, we find that the best exponent pair in $H$ is $(\theta + \varepsilon, \frac{1}{2} + \theta + \varepsilon)$. Making this choice, and using $Z^\varepsilon < T^\varepsilon$, we have 
\[
S \ll_{\varepsilon} T^{\varepsilon}(TZ^{-6} + T^{2 + 1/(2\theta)}Z^{-(8 + 4/\theta)}).
\]
If $Z > T^{(2\theta + 1)/(4\theta + 8)}$ then $TZ^{-6} \gg T^{2 + 1/(2\theta)}Z^{-(8 + 4/\theta)}$ and hence
\[
S\ll_{\varepsilon} T^{1 + \varepsilon}Z^{-6}.
\]
It follows from $|\zeta(1/2 + i\tau_s)| \ll Z \le T^{\theta + \varepsilon}$ that
\[
\sum_{\tau_s}|\zeta(1/2 + i\tau_s)|^A \ll_{\varepsilon}T^{1 + \varepsilon}Z^{A-6} \ll_{\varepsilon} T^{1 + \theta(A - 6) + \varepsilon}. 
\]

On the other hand if $Z \le T^{(2\theta + 1)/(4\theta + 8)}$ then $S \ll_{\varepsilon} T^{2 + 1/(2\theta) + \varepsilon}Z^{-(8 + 4/\theta)}$. Therefore, using $Z \ll T^{\theta + \varepsilon}$,
\[
\sum_{\tau_s}|\zeta(1/2 + i\tau_s)|^A \ll_{\varepsilon}T^{2 + 1/(2\theta) + \varepsilon}Z^{A - (8 + 4/\theta)} \ll_{\varepsilon} T^{1 + \theta(A - 6) + \varepsilon}
\]
provided that $A \ge 8 + 4/\theta$ and 
\[
\frac{1}{2\theta} - 2\theta \le 3
\]
which is satisfied for all $\theta \ge (\sqrt{13} - 3)/4$. Therefore, in both cases we have (by taking $Z = T/2, T/4, T/8, \ldots$), 
\[
\sum_{r \le R}|\zeta(1/2 + it_r)|^A = \sum_{Z}\sum_{\tau_s}|\zeta(1/2 + i\tau_s)|^A \ll_{\varepsilon} \log T\cdot T^{1 + \theta(A - 6) + \varepsilon}.
\]
The claimed result therefore follows from the equivalence of \eqref{equiv_formula_1} and \eqref{equiv_formula_2}.
\end{proof}

\subsection{Proof of Theorem \ref{zeta_moment_thm}}
Ivi\'{c} and Zhai \cite{ivic_mean_2012} showed that if $(k, \ell)$ is an exponent pair satisfying $\ell + (2j - 1)k < 1$, then 
\[
\int_1^T|\zeta(1/2 + it)|^4|\zeta(\sigma + it)|^{2j}\text{d}t \ll_{\varepsilon} T^{1 + \varepsilon},\qquad \sigma > \frac{\ell - k + 6jk}{1 + 4j k}.
\]
In particular, to establish Theorem \ref{zeta_moment_thm} we take $j = 2$ and search for favourable exponent pairs by solving the optimisation problem 
\[
\min_{(k, \ell) \in H}\frac{\ell + 11k}{1 + 8 k}\qquad \text{s.t.}\qquad \ell + 3k < 1.
\]
The solution is $(k, \ell) = (p_5, q_5) = (\frac{1}{56}, \frac{127}{140} + \varepsilon)$, which gives the desired result.

\subsection{Proof of Theorem \ref{mu_est_thm}}
We begin with the standard argument that if $\ell \ge k + \sigma$, then $\zeta(\sigma + it) \ll_{\varepsilon} t^{(k + \ell - \sigma)/2 + \varepsilon}$, reproduced below for completeness. From the approximate functional equation for $\zeta(s)$ \cite{hardy_zeros_1921},
\[
\zeta(s) = \sum_{n \le \sqrt{t/(2\pi)}}n^{-s} + \chi(1 - s)\sum_{m \le \sqrt{t/(2\pi)}}m^{s - 1} + o(1),\qquad 1/2 \le \sigma < 1.
\]
where $\chi(1 - s) \ll t^{1/2 - \sigma}$. If $(k, \ell)$ is an exponent pair satisfying $\ell - k - \sigma \ge 0$, then for any $1 \le N \ll t^{1/2}$ and $\sigma \ge 1/2$,
\[
\sum_{N < n \le 2N}n^{-\sigma - it} \ll N^{-\sigma}\left(\frac{t}{N}\right)^k N^{\ell} = N^{\ell - k -\sigma}t^k \ll t^{(k + \ell - \sigma)/2},
\]
\begin{align*}
\chi(1 - s)\sum_{N < n \le 2N}n^{-1 + \sigma + it} \ll t^{1/2 - \sigma}N^{\sigma - 1}\left(\frac{t}{N}\right)^kN^\ell \ll t^{(k + \ell - \sigma)/2},
\end{align*}
and hence, via a dyadic division,
\begin{equation}\label{upper_bound_zeta_temp}
\zeta(\sigma + it) \ll t^{(k + \ell - \sigma)/2}\log t,\qquad \ell - k \ge \sigma \ge 1/2.
\end{equation}
Therefore, the optimisation problem we consider is 
\[
\min_{(k, \ell) \in H}\frac{k + \ell - \sigma}{2}\qquad \text{s.t.}\qquad \ell - k \ge \sigma. 
\]
The solution lies on the boundary of $H$ and we have 
\[
\mu(\ell - k) \le k.
\]
Substituting points of the form $(k_n, \ell_n)$, we obtain, for $0 \le n \le 10$, that
\[
\mu( \tfrac{1}{2} ) \leq \tfrac{13}{84},\qquad
\mu( \tfrac{88225}{153852} ) \leq \tfrac{4742}{38463},\qquad
\mu( \tfrac{521}{796} ) \leq \tfrac{18}{199},\qquad
\mu( \tfrac{53141}{76066} ) \leq \tfrac{2779}{38033},
\]
\[
\mu( \tfrac{3620}{5119} ) \leq \tfrac{715}{10238},\qquad
\mu( \tfrac{52209}{69128} ) \leq \tfrac{2371}{43205},\qquad
\mu( \tfrac{1389}{1736} ) \leq \tfrac{9}{217},\qquad
\mu( \tfrac{134765}{163248} ) \leq \tfrac{2779}{81624},
\]
\[
\mu( \tfrac{18193}{21906} ) \leq \tfrac{715}{21906},\qquad
\mu( \tfrac{249}{280} ) \leq \tfrac{1}{56},\qquad
\mu( \tfrac{277}{300} ) \leq \tfrac{1}{100}.
\]
Theorem \ref{mu_est_thm} then follows from the convexity property of $\mu(\sigma)$. Specifically, for any fixed $\sigma_1 < \sigma_2$, 
\[
\mu(\sigma) \le \frac{(\sigma_2 - \sigma)\mu(\sigma_1) + (\sigma - \sigma_1)\mu(\sigma_2)}{\sigma_2 - \sigma_1},\qquad (\sigma_1 \le \sigma \le \sigma_2).
\]

\subsection{Proof of Theorem \ref{zeta_bound_thm}}
Taking the exponent pair $(p_n, q_n) = (k_{n + 4}, \ell_{n + 4})$ for $n \ge 5$ and choosing $\sigma = \sigma_n := q_n - p_n$ in \eqref{upper_bound_zeta_temp}, we have $\mu(\sigma_n) \le \mu_n$, where 
\[
\sigma_n = 1 - \frac{3n^2 - 3n + 2}{n(n - 1)^2(n + 2)},\qquad \mu_n = \frac{2}{(n - 1)^2(n + 2)},\qquad (n \ge 5).
\]
Using the convexity of $\mu(\sigma)$, we have 
\begin{equation}\label{mu_B_bound}
\mu(\lambda\sigma_n + (1 - \lambda)\sigma_{n + 1}) \le \lambda\mu_n + (1 - \lambda)\mu_{n + 1} \le B(1 - \lambda\sigma_n - (1 - \lambda)\sigma_{n + 1})^{3/2}
\end{equation}
for all $\lambda \in [0, 1]$ and $n \ge 5$, where
\[
B := \max_{n \ge 5, \lambda \in [0, 1]}f(n, \lambda),\qquad f(n, \lambda) := \frac{\lambda\mu_n + (1 - \lambda)\mu_{n + 1}}{(1 - \lambda\sigma_n - (1 - \lambda)\sigma_{n + 1})^{3/2}}.
\]
For each fixed $n$, the function $f(n, \lambda)$ is maximised by 
\[
\lambda = \lambda_n := \frac{2\mu_n\sigma_{n + 1} - 2\mu_n - 3 \mu_{n + 1}\sigma_n + \mu_{n + 1}\sigma_{n + 1} + 2\mu_{n + 1}}{(\mu_n - \mu_{n + 1}) (\sigma_n - \sigma_{n + 1})}.
\]
Thus, for all $\lambda \in [0, 1]$ and $n \ge 5$,
\begin{align*}
f(n, \lambda) \le f(n, \lambda_n) &= \frac{2 (\mu_{n + 1} - \mu_n)^{3/2}}{3\sqrt{3} (\sigma_n - \sigma_{n + 1})(\mu_{n + 1} - \mu_n + \mu_n\sigma_{n + 1} - \mu_{n + 1}\sigma_n)^{1/2}} \\
&= \frac{2}{3^{3/2}}\frac{n^{1/2} (n + 1)^{3/2}}{n^2 + 1}.
\end{align*}
Note that the RHS is decreasing for $n \ge 5$, and $f(5, \lambda_5) = 2\sqrt{10}/13$. Hence, from \eqref{mu_B_bound},
\[
\mu(\sigma) \le \frac{2\sqrt{10}}{13}(1 - \sigma)^{3/2},\qquad \sigma \ge \sigma_5 = \frac{249}{280}.
\]
For $\sigma < 249/280$, the desired result follows from Theorem \ref{mu_est_thm}.

\subsection{Proof of Theorem \ref{zeta_bound_thm1}}
Using the same argument as in the proof of Theorem \ref{zeta_bound_thm}, we have 
\begin{equation}\label{mu_sigma_bound}
\mu(\sigma) \le \frac{2}{3^{3/2}}\frac{n^{1/2} (n + 1)^{3/2}}{n^2 + 1}(1 - \sigma)^{3/2},\qquad \sigma_n \le \sigma \le \sigma_{n + 1}.
\end{equation}
For $n \ge 33$, we have
\begin{align*}
\frac{2}{3^{3/2}}\frac{n^{1/2} (n + 1)^{3/2}}{n^2 + 1} &< \frac{2}{3^{3/2}} + \left(\frac{1}{3} + \frac{1}{100}\right)\left(\frac{3n^2 + 3n + 2}{n^2(n + 1)(n + 3)}\right)^{1/2} \\
&= \frac{2}{3^{3/2}} + \left(\frac{1}{3} + \frac{1}{100}\right)(1 - \sigma_{n + 1})^{1/2}
\end{align*}
where the inequality is verified by a routine calculation. In fact, we may replace the constant $1/3 + 1/100$ with $1/3 + \varepsilon$ for any $\varepsilon > 0$, provided we take $n$ sufficiently large (depending on $\varepsilon$). Therefore, for $\sigma \ge \sigma_{33} = 1 - 317/118272$, we have 
\[
\mu(\sigma) \le \left(\frac{2}{3^{3/2}} + \frac{103}{300}(1 - \sigma)^{1/2}\right)(1 - \sigma)^{3/2}
\]
as required.

\subsection{Proof of Corollary \ref{zero_density_bound_thm1}}
The results of this section depend on the following well-known lemma, due originally to Montgomery \cite[Thm.\ 12.3]{montgomery_topics_1971}. 

\begin{lemma}[Montgomery \cite{montgomery_topics_1971}]\label{monty_lemma_12_3}
Let
\[
M(\alpha, T) := \max_{\substack{\sigma \ge \alpha\\1 \le t \le T}}|\zeta(\sigma + it)|.
\]
Then, we have
\[
N(\sigma, T) \ll (M(\alpha, 8T) \log^5T)^{2(1 - \sigma)(3\sigma - 1 - 2\alpha)/((2\sigma - 1 - \alpha)(\sigma - \alpha))}\log^8T. 
\]
for all $1/2 \le \alpha \le 1$ and $\sigma \ge (\alpha + 1)/2$.
\end{lemma}

This allows us to easily translate bounds on $\mu(\sigma)$ into zero-density estimates close to $\sigma = 1$. However, we will work directly with exponent pairs to illustrate the underlying optimisation problem. If $(k, \ell)$ is an exponent pair, then by the approximate functional equation
\[
M(\ell - k) \ll_{\varepsilon} T^{k + \varepsilon}. 
\]
Therefore, we set $\alpha = \ell - k$ and consider (for each $\sigma$) the optimisation problem 
\begin{equation}\label{zero_den_optimisation_problem}
\min_{(k, \ell) \in H}\frac{k(3\sigma - 1 + 2k - 2\ell)}{(2\sigma - 1 + k - \ell)(\sigma + k - \ell)}\qquad \text{s.t.}\qquad \frac{1}{2} \le \ell - k \le 1,\, \sigma \ge \frac{\ell - k + 1}{2}.
\end{equation}
Although the solution $(k(\sigma), \ell(\sigma))$ varies smoothly with $\sigma$, we find numerically that the following choices are near-optimal: for $0 \le n \le 4$, we choose
\[
(k, \ell) = (k_{n + 4}, \ell_{n + 4}),\qquad (\sigma_n \le \sigma < \sigma_{n + 1})
\]
where $(k_m, \ell_m)$ is defined in \eqref{knln_defn}, and 
\begin{alignat*}{3}
    \sigma_0 &= 9/10,\qquad  &&\sigma_1 = 0.9573\ldots,\qquad &&\sigma_2 = 0.9621\ldots,\\
    \sigma_3 &= 0.9644\ldots,\qquad &&\sigma_4 = 0.9669\ldots,\qquad  &&\sigma_5 = 1.
\end{alignat*}
Theorem \ref{zero_density_bound_thm1} follows from substituting the values of $(k, \ell)$ for each range of $\sigma$ into $\alpha = \ell - k$ and using Lemma \ref{monty_lemma_12_3}. For instance, in the case $n = 0$ we take, upon ignoring $\varepsilon$'s for ease of presentation, $(k, \ell) = (k_4, \ell_4) = (715/10238, 7955/10238)$ and hence
\[
f(\sigma) \le \frac{715(15357 \sigma - 12359)(1 - \sigma)}{(5119 \sigma - 3620)(10238 \sigma - 8739)},\qquad \frac{9}{10} \le \sigma < \sigma_1. 
\]
The value of $\sigma_n$ for $n \ge 1$ is the ``crossover" point between the bounds on $f(\sigma)$ arising from the exponent pairs $(k_{n + 4}, \ell_{n + 4})$ and $(k_{n + 5}, \ell_{n + 5})$ respectively. For instance, $\sigma = \sigma_1 = 0.9573\ldots$ solves
\[
\frac{715(15357 \sigma - 12359)(1 - \sigma)}{(5119 \sigma - 3620)(10238 \sigma - 8739)} = \frac{75872(103692 \sigma - 86773)(1 - \sigma)}{5(69128\sigma - 52209) (138256\sigma - 121337)}.
\]
\begin{remark}
It is possible to show a slightly stronger result in small ranges of $\sigma$. Instead of choosing $(k, \ell)$ from the vertices of $H$, we consider all exponent pairs along the boundary of $H$. This gives an improvement for values of $\sigma$ near $\sigma_n$. The resulting bounds on $f(\sigma)$ are unwieldy expressions so we will instead provide a numerical example. For $\sigma = 45/47$ (chosen to be close to $\sigma_1 = 0.9573\ldots$), we take 
\[
(k, \ell) = \left(\lambda k_4 + (1 - \lambda) k_5, \lambda \ell_4 + (1 - \lambda)\ell_5\right)
\]
in \eqref{zero_den_optimisation_problem} with $\lambda$ chosen optimally as
\[
\lambda = \frac{353866232 \sqrt{2674870481950895} - 15254503135395337}{9032663480578423} = 0.3373\ldots \in [0, 1]. 
\]
This gives $A(45/47) \le 1.2303$, which improves on Corollary \ref{zero_density_bound_thm1}.
\end{remark}
\subsection{Proof of Corollary \ref{zero_density_bound_thm}}
Taking $\alpha = k\sigma - (k - 1)$ for some $k > 1$ in Lemma \ref{monty_lemma_12_3}, and using Theorem \ref{zeta_bound_thm} to estimate $M(\alpha, 8T)$, we obtain, for $1 - 1/(2k) \le \sigma \le 1$,
\[
N(\sigma, T) \ll_{\varepsilon} T^{c(1 - \sigma)^{3/2} + \varepsilon},\qquad c = \frac{k^{3/2}(4k - 6)}{k^2 - 3k + 2}B,
\]
where $B = 2\sqrt{10}/13$ as in Theorem \ref{zeta_bound_thm}. Note that $c$ is minimised by the choice $k = 4.928408\ldots$, which gives $c \le 6.3453$, as required. 

\subsection{Proof of Theorem \ref{divisor_problem_thm}}
Let $m(\sigma)$ be as defined in \eqref{m_sigma_defn}. In the standard treatment (see e.g.\ Ivi\'{c} \cite[\S 13.3]{ivic_riemann_2003}), if $m(\sigma) \ge n$ then $\Delta_n(x) \ll_{\varepsilon} x^{\sigma + \varepsilon}$. Thus the problem reduces to estimates of $m(\sigma)$, which in turn depend on a certain large values estimate of $\zeta(\sigma + it)$. To this end, let $0 < V \le T$ and $t_r$ ($1 \le r \le R)$ be a set of points satisfying
\[
|\zeta(\sigma + it)| \ge V,\qquad |t_r| \le T,\qquad (1 \le r \le R),
\]
\[
|t_r - t_s| \ge \log^4T,\qquad (1 \le r \ne s \le R)
\]
and furthermore suppose that $\mu(\theta) \le c(\theta)$ for $1/2 \le \theta \le 1$, where $c(\theta)$ is the piecewise-defined function in Theorem \ref{mu_est_thm}. 
Following the argument in \cite[Lem.\ 8.2]{ivic_riemann_2003}, let $\theta = \theta(\sigma)$ be implicitly defined by 
\[
2c(\theta) + 1 + \theta - 2(1 + c(\theta))\sigma = 0.
\]
Suppose that for a particular value of $\sigma$, $\theta(\sigma) \in [\sigma_1, \sigma_2]$ and that 
\[
c(\theta) = A + B\theta,\qquad (\sigma_1 \le \theta \le \sigma_2),
\]
so that 
\[
\theta = \frac{2\sigma - 1 - 2A(1 - \sigma)}{2B(1 - \sigma) + 1}.
\]
Furthermore, let 
\[
f(\theta) := \frac{2(1 + c(\theta))}{c(\theta)} = \frac{2(A + B + 1)}{A + B(2\sigma - 1)}
\]
then, following the argument leading up to Ivi\'{c} \cite[Eqn.\ (8.97)]{ivic_riemann_2003}, for any exponent pair $(k, \ell)$ we have   
\begin{equation}\label{m_bound_R_bound}
\begin{split}
R &\ll T^{\varepsilon}(TV^{-2f(\sigma)} + T^{(4 - 4\sigma)/(1 + 2\sigma)}V^{-12/(1 + 2\sigma)} \\
&\qquad + T^{4(1 - \sigma)(k + \ell)/((2 + 4\ell)\sigma - 1 + 2k - 2\ell)}V^{-4(1 + 2k + 2\ell)/((2 + 4\ell)\sigma - 1 + 2k - 2\ell)})
\end{split}
\end{equation}
for $\theta(\sigma_1) \le \sigma \le \theta(\sigma_2)$. For each integer $n$, we seek to find the smallest $\sigma = \sigma(n)$ for which $R \ll_{\varepsilon} T^{1 + \varepsilon}V^{-n}$, since this implies $m(\sigma) \ge n$ and $\Delta_n(x) \ll_{\varepsilon} x^{\sigma + \varepsilon}$. Assuming $\sigma_1 \le \sigma \le \sigma_2$, we use $V \ll T^{c(\sigma)}$ to compute
\begin{equation}\label{A1_estimate}
T^{4( - \sigma)/(1 + 2\sigma)}V^{-12/(1 + 2\sigma)} \ll TV^{-A_1},\qquad A_1(\sigma) = \frac{1}{1 + 2\sigma}\left(12 + \frac{3(2\sigma - 1)}{c(\sigma)}\right)
\end{equation}
and also 
\[
T^{4(1 - \sigma)(k + \ell)/((2 + 4\ell)\sigma - 1 + 2k - 2\ell)}V^{-4(1 + 2k + 2\ell)/((2 + 4\ell)\sigma - 1 + 2k - 2\ell)} \ll TV^{-A_2},
\]
\begin{equation}\label{A2_estimate}
A_2(\sigma) = \frac{4c(\sigma)(1 + 2k + 2\ell) + 2\sigma(1 + 2k + 4\ell)  - 1 - 2k - 6\ell }{c(\sigma)(2 k + (2\ell + 1) (2\sigma - 1))}.
\end{equation}
Therefore we consider the optimisation problem (for fixed $n$)
\[
\min_{1/2 \le \sigma \le 1}\sigma,\qquad \text{s.t.}\qquad (k, \ell) \in H,\; A_1(\sigma) \ge n, \; A_2(\sigma) \ge n,
\]
or, equivalently, the dual problem (for fixed $1/2 \le \sigma \le 1$)
\[
\max_{(k, \ell)\in H}\min\left\{A_1(\sigma), A_2(\sigma)\right\}.
\]
We use the values of $c(\sigma)$ from Theorem \ref{mu_est_thm} so the optimisation problem is well-defined. In the range $0.646 \le \sigma \le 0.794$, we numerically compute the solution as
\begin{equation}\label{divisor_problem_exp_pairs}
(k, \ell) = \begin{cases}
(k_{-4}, \ell_{-4}),& 0.646 \le \sigma \le 0.722,\\
(k_{-3}, \ell_{-3}),& 0.723 \le \sigma \le 0.765,\\
(k_{-2}, \ell_{-2}),& 0.766 \le \sigma \le 0.794,
\end{cases}
\end{equation}
where $(k_n, \ell_n)$ are defined in \eqref{knln_defn}. The range of $\sigma$ was chosen to obtain estimates for $\Delta_n(x)$ for $9 \le n \le 20$; estimates for larger $n$ can be obtained by extending the range for $\sigma$. Substituting \eqref{divisor_problem_exp_pairs} into \eqref{A1_estimate} and \eqref{A2_estimate}, and taking $c(\sigma)$ from Theorem \ref{mu_est_thm}, we obtain 
\begin{equation}\label{m_sigma_lower_bound}
m(\sigma)\geq\begin{cases}
	 \frac{8 (453710742 - 1311814001 \sigma)}{(21906 \sigma - 8117) (251324 \sigma - 220633)} ,& 0.646 \leq\sigma\leq \frac{521}{796} ,\\
	 \frac{23850077 - 66940702 \sigma}{(1508 \sigma - 1333) (21906 \sigma - 8117)} ,& \frac{521}{796} \leq\sigma\leq \frac{53141}{76066} ,\\
	 \frac{2 (4130567 - 11066434 \sigma)}{(454 \sigma - 405) (21906 \sigma - 8117)} ,& \frac{53141}{76066} \leq\sigma\leq \frac{3620}{5119} ,\\
	 \frac{6 (268525549815 - 626275790894 \sigma)}{(21906 \sigma - 8117) (52938216 \sigma - 49318855)} ,& \frac{3620}{5119} \leq\sigma\leq 0.722 ,\\
	 \frac{30 (200973859502 - 466361285421 \sigma)}{(81624 \sigma - 30479) (52938216 \sigma - 49318855)} ,& 0.723 \leq\sigma\leq \frac{52209}{69128} ,\\
	 \frac{10 (6283940958 - 14261159585 \sigma)}{(81624 \sigma - 30479) (502648 \sigma - 471957)} ,& \frac{52209}{69128} \leq\sigma\leq 0.765 ,\\
	 \frac{2 (681633153 - 1510627522 \sigma)}{(1736 \sigma - 673) (502648 \sigma - 471957)} ,& 0.766 \leq\sigma\leq 0.794
\end{cases}
\end{equation}
Estimates for $\Delta_n(x)$ can then be found by inverting these relations. For instance, inverting the first case of \eqref{m_sigma_lower_bound} gives (with the aid of the symbolic algebra package \texttt{SymPy} \cite{sympy_2017})
\[
\alpha_n \le \frac{3436591703 n - 5247256004 + \sqrt{D}}{5505503544 n},\qquad 8.957\le n \le \frac{413385287}{44567046},
\]
where
\begin{align*}
D &= 1950477021421092025\, n^{2} - 16082104109471712440\, n \\
&\qquad + 27533695571514048016.
\end{align*}
This allows us to compute $\alpha_9 \le 0.6472$. 
\begin{comment}
\subsection{Proof of Theorem \ref{gamma_gaps_thm}}
Ivi\'{c} \cite[Thm.\ 10.2]{ivic_riemann_2003} showed that 
\[
\gamma_{n + 1} - \gamma_n \ll_{\varepsilon} \gamma_n^{\theta + \varepsilon},\qquad \theta = \frac{k + \ell}{4k + 4\ell + 2}
\]
so we consider the following optimisation problem 
\[
\min_{(k, \ell) \in H}\frac{k + \ell}{4k + 4\ell + 2}
\]
Since the objective function is an increasing function of $k + \ell$, it suffices to find $(k, \ell) \in H$ that minimises $k + \ell$. The solution to this linear program is $(k, \ell) = (13/84 + \varepsilon, 55/84 + \varepsilon)$, which gives $\theta = 17/110 = 0.1545\ldots$. 
\end{comment}

\subsection{Proof of Theorem \ref{pythagorean_triple_thm}}
Following the argument of \cite{menzer_number_1986}, we have that 
\[
R(N) \ll_{\varepsilon} N^{\varepsilon}(N^{1/2}B^{-3/2} + N^{1/3}B^{-5/6} + N^{(k + \ell - 1/2)/2}B^{-2(k + \ell - 1)}).
\]
Balancing the first and last terms, we choose 
\[
B = N^{(k + \ell - 3/2)/(4(k + \ell) - 7)}
\]
to obtain $R(N) \ll_{\varepsilon} N^{\theta + \varepsilon}$, where
\[
\theta(k, \ell) = \max\left\{\frac{1}{3} - \frac{5}{6}\frac{k + \ell - 3/2}{4(k + \ell) - 7}, \frac{1}{2} - \frac{3}{2}\frac{k + \ell - 3/2}{4(k + \ell) - 7}\right\},
\]
Both terms are increasing in $k + \ell$, so the solution of 
\[
\min_{(k, \ell)\in H}\theta(k, \ell)
\]
is given by $(k, \ell) = (k_0, \ell_0) = (13/84 + \varepsilon, 55/84 + \varepsilon)$. This gives
\[
R(N) \ll_{\varepsilon} N^{71 / 316 + \varepsilon},\qquad \frac{71}{316} = 0.22468\ldots,
\]
as required. 

\section{Conclusion and future work}
As a concluding remark we speculate how some possible additions to the set of known exponent pairs will affect the convex hull $H$. Further refinements to the Bombieri--Iwaniec method, useful for bounding \eqref{exp_pair_defn} for $\log N / \log (yN^{-\sigma})$ close to $1/2$, can possibly generate better exponent pairs of the type \eqref{BI_exp_pair} which lie on the line of symmetry $\ell = k + 1/2$. By lowering the value of $\theta$, the hull is expanded inwards towards $(0, 1/2)$, a point which, if obtained, represents the ultimate achievement in this regard (and proves the exponent pair conjecture). 

On the other extreme, refinements to the $m$th derivative test, for large $m$, has the effect of widening the hull close to the points $(0, 1)$ and $(1/2, 1/2)$, so that the boundary of $H$ gets closer to the coordinate axes $\ell$ and $k$ respectively. Improvements in this result lead to progress in results such as Theorem \ref{M_bound_12_thm}, Theorem \ref{zeta_bound_thm} and Corollary \ref{zero_density_bound_thm}. 

An interesting intermediate case are the $m$th derivative tests for small $m$. Further refinements of these methods lead to new exponent pairs along the lines 
\[
\ell = 1 - (m - 1)k.
\]
In the case $m = 4$, a notable hypothetical exponent pair is $(1/12 + \varepsilon, 3/4 + \varepsilon)$. So far, a number of results have been established that are of the same strength over certain ranges. For instance, \cite[Thm.~1]{robert_fourth_2016} implies there exists $P$, $c > 0$ such that, for $f \in \textbf{F}(N, P, \sigma, y, c)$,   
\[
\sum_{N < n \le 2N}e(f(n)) \ll_{\varepsilon} \left(\frac{y}{N^\sigma}\right)^{1/12}N^{3/4 + \varepsilon} + N^{11/12 + \varepsilon},\qquad (yN^{-\sigma} \gg N^{4/3})
\]
which for $yN^{-\sigma} \gg N^2$ implies the same bound as a hypothetical $(1/12, 3/4 + \varepsilon)$ exponent pair. This particular exponent pair also represents the limit of certain methods. For instance, it follows from the work of Sargos \cite[Thm.\ 7.1]{sargos_points_1995} that if $(k, \ell)$ is an exponent pair, then 
\[
\sum_{N < n \le 2N}e(f(n)) \ll_{\varepsilon} \left(\frac{y}{N^{\sigma}}\right)^{k_1}N^{\ell_1} + \left(\frac{y}{N^\sigma}\right)^{1/12 + \varepsilon}N^{3/4 + \varepsilon}
\]
with
\[
k_1 = \frac{5k + \ell + 2}{8(5k + 3\ell + 2)},\qquad \ell_1 = \frac{29k + 21\ell + 10}{8(5k + 3\ell + 2)}.
\]
This represents a new process for obtaining novel exponent pairs, up to $(1/12 + \varepsilon, 3/4 + \varepsilon)$.
\section*{Acknowledgements}
We thank T. Oliveira e Silva for discussions on an earlier version of this work. Many thanks to G. Debruyne and T. Tao for spotting some errors in our preprint. Additionally, we thank O.~Bordell\`{e}s, D. R.~Heath-Brown, M.~Huxley, B.~Kerr, O.~Ramar\'{e}, I.~Shparlinski, N.~Watt, A.~Weingartner and T.~Wooley for their kind feedback upon the first preprint of this article. Last but not least we would like to express our gratitude to the anonymous referee for multiple helpful suggestions. 

\section*{Program 1}\label{verify_program}
\lstdefinestyle{mystyle}{
    numberstyle=\tiny\color{gray},
    basicstyle=\ttfamily\footnotesize,
    breakatwhitespace=false,         
    breaklines=true,                 
    captionpos=b,                    
    keepspaces=true,                 
    numbers=left,                    
    numbersep=5pt,                  
    showspaces=false,                
    showstringspaces=false,
    showtabs=false,                  
    tabsize=2
}
\lstset{style=mystyle}

\begin{lstlisting}[language=Python]
from fractions import Fraction as F
from math import log

# The van der Corput A transform
def A(k, l):
    return (k / (2 * k + 2), (k + l + 1) / (2 * k + 2))

# The van der Corput B transform
def B(k, l):
    return (l - F(1, 2), k + F(1, 2))

def C(k, l):
    return (k / (12 * (1 + 4 * k)), l / (12 * (1 + 4 * k)) + F(11, 12))
    
# returns (k_n, \ell_n)
def point(n):
    if n < 0:
        (k, l) = point(-n)
        return B(k, l)
    if n == 0: 
        return (F(13, 84), F(55, 84))
    elif n == 1:
        return (F(4742, 38463), F(35731, 51284))
    elif n == 2:
        return (F(18, 199), F(593, 796))
    elif n == 3:
        return (F(2779, 38033), F(58699, 76066))
    elif n == 4:
        return (F(715, 10238), F(7955, 10238))
    elif n == 5: 
        return A(F(4742, 38463), F(35731, 51284))
    elif n == 6: 
        return A(F(18, 199), F(593, 796))
    elif n == 7: 
        return A(F(2779, 38033), F(58699, 76066))
    elif n == 8: 
        return A(F(715, 10238), F(7955, 10238))
    
    m = n - 4
    return (F(2, (m - 1) * (m - 1) * (m + 2)), 1 - F(3 * m - 2, m * (m - 1) * (m + 2)))
    
# returns whether point p = (k, l) lies in $H_N$
def in_hull(p, N):
    (k, l) = p
    
    # check if p lies above line joining p_n, p_{n + 1}
    for n in range(-N, N):
        (k1, l1) = point(n)
        (k2, l2) = point(n + 1)
        if k * (l2 - l1) + l * (k1 - k2) < k1 * l2 - l1 * k2:
            return False
    
    (kN, lN) = point(N)
    if k * (1 - lN) + l * (kN - 0) < kN * 1 - lN * 0:
        return False
    
    (k_N, l_N) = point(-N)
    if k * (l_N - F(1, 2)) + l * (F(1, 2) - k_N) < F(1, 2) * l_N - F(1, 2) * k_N:
        return False
    
    if k + l > 1:
        return False
    
    return True

# verify Lemma 3.6
def verify_lemma_3_6():
    N = 300
    
    # check exponent pairs of the form (1.3)
    thetas = [F(9, 56), F(89, 560), F(89, 570), F(32, 205), F(13, 84)]
    for th in thetas:
        assert in_hull((th, th + F(1, 2)), N)
    
    # check exponent pairs of the form (1.4)
    assert in_hull((F(2, 13), F(35, 52)), N)
    assert in_hull((F(516247, 6629696), F(5080955, 6629696)), N)
    assert in_hull((F(6299, 43860), F(29507, 43860)), N)
    assert in_hull((F(771, 8116), F(1499, 2029)), N)
    assert in_hull((F(21, 232), F(173, 232)), N)
    assert in_hull((F(1959, 21656), F(16135, 21656)), N)
    
    # check exponent pairs of the form (1.6)
    mvartheta = [(4, F(1, 13)),
                 (8, F(1, 204)), 
                 (9, F(7, 2640)), (9, F(1, 360)),
                 (10, F(1, 716)), (10, F(1, 649)), (10, F(7, 4540)), (10, F(1, 615)),
                 (11, F(1, 915))]
    for (m, vartheta) in mvartheta:
        assert in_hull((vartheta, 1 - (m - 1) * vartheta), N)
    
    # check exponent pairs of the form (1.5) (m \leq 100)
    for m in range(101):
        assert in_hull((F(169, 1424 * (2 ** m) - 338), 1 - F(169, 1424 * (2 ** m) - 338) * F(712 * m + 1577, 712)), N)
    
    # check exponent pairs of the form (1.7) (m \leq 100)
    for m in range(3, 101):
        assert in_hull((1 / (25 * m * m * (m - 1) * log(m)), 1 - 1 / (25 * m * m * log(m))), N)

    # check exponent pairs of the form (1.8) (m \leq 100)
    for m in range(3, 101):
        assert in_hull((F(2, (m + 2) * (m - 1) ** 2), 1 - F(3 * m - 2, m * (m - 1) * (m + 2))), N)
        

# verify Lemma 4.3
def verify_lemma_4_3():
    # check that A(k_m, \ell_m) \in H_{1000} for |m| < 100
    for m in range(-99, 100):
        (k_m, l_m) = point(m)
        assert in_hull(A(k_m, l_m), 1000)
        assert in_hull(C(k_m, l_m), 1000)
        

verify_lemma_3_6()
verify_lemma_4_3()
\end{lstlisting}

\newpage
\printbibliography
\end{document}